\documentclass[11pt,letterpaper, reqno]{amsart}
\usepackage[
backend=biber,
style=alphabetic,
sorting=nyt,
maxbibnames=9
]{biblatex}
\usepackage{courier}
\usepackage{amsmath}
\usepackage{mathrsfs}
\usepackage{amssymb}
\usepackage{adjustbox}
\usepackage{amsthm}
\usepackage{stmaryrd}
\usepackage{xcolor}
\usepackage[hidelinks]{hyperref}
\usepackage{mathtools}
\usepackage{tikz-cd}
\usepackage{amsmath,amscd}
\usepackage{blindtext}
\usepackage{relsize}
\usepackage{mathbbol}
\usepackage{amsfonts}
\DeclareSymbolFontAlphabet{\mathbb}{AMSb} 
\DeclareSymbolFontAlphabet{\mathbbl}{bbold}
\newcommand{\Prism}{{\mathlarger{\mathbbl{\Delta}}}}
\newtheorem{theorem}{Theorem}
\newtheorem{corollary}{Corollary}[theorem]
\newtheorem{lemma}[theorem]{Lemma}
\newtheorem{definition}[theorem]{Definition}
\newtheorem{proposition}[theorem]{Proposition}
\newtheorem{remark}[theorem]{Remark}

\theoremstyle{definition}

\newtheorem{convention}[theorem]{Convention}
\newcommand{\Tatem}{\widehat{T_{\mathfrak{m}}}}

\newcommand{\Ainf}{A_{\mathrm{inf}}}
\newcommand{\AinfR}{A_{\mathrm{inf}}(R)}
\newcommand{\AcrisR}{A_{\mathrm{cris}}(R)}
\newcommand{\Acris}{A_{\mathrm{cris}}}
\newcommand{\Bcris}{B_{\mathrm{cris}}}
\newcommand{\OAcris}{\mathrm{OA}_{\mathrm{cris}}}
\newcommand{\OBcris}{\mathrm{OB}_{\mathrm{cris}}}

\newcommand{\Tcris}{T_{\mathrm{cris}}}
\newcommand{\Hom}{\mathrm{Hom}}
\newcommand{\GR}{\mathcal{R}}
\newcommand{\GS}{\mathfrak{S}}
\newcommand{\GM}{\mathfrak{M}}
\newcommand{\SM}{\mathscr{M}}
\newcommand{\SD}{\mathscr{D}}
\newcommand{\Fil}{\mathrm{Fil}}
\newcommand{\Tor}{\mathrm{Tor}}
\newcommand{\etalering}{\mathcal{O}_{\mathcal{E}}}
\newcommand{\cris}{\mathrm{cris}}
\newcommand{\GST}{\GS^{(1)}}

\newcommand{\Dcris}{D_{\mathrm{cris}}}
\newcommand{\Dcrisdual}{D_{\mathrm{cris}}^{\vee}}
\newcommand{\F}{\mathrm{F}}
\newcommand{\T}{\mathcal{T}}
\newcommand{\hatM}{\widehat{M}}

\begin{document}
\title{Fontaine-Laffaille Theory over Power Series Rings}
\author{Christian Hokaj}
\begin{abstract}
Let $k$ be a perfect field of characteristic $p > 2$. We extend the equivalence of categories between Fontaine-Laffaille modules and $\mathbb{Z}_p$ lattices inside crystalline representations with Hodge-Tate weights at most $p-2$ of \cite{Fontaine} to the situation where the base ring is the power series ring over the Witt vectors $R \coloneqq W(k)\llbracket t_1, \cdots , t_d\rrbracket$ and $T$ a $p$-adically complete ring that is \'etale over the Tate Algebra $W(k)\langle t_1^{\pm 1}, \cdots , t_d^{\pm 1}\rangle$. 
\end{abstract}
\maketitle
\tableofcontents

\begin{section}{Introduction}
Let $k$ be a perfect field of characteristic $p > 2$, $W(k)$ the ring of Witt-vectors over $k$ with $\varphi$ its Frobenius automorphism, $K$ its fraction field, and $R \coloneqq W(k)\llbracket t_1, \cdots , t_d \rrbracket$ the ring of power series over $W(k)$ in $d$ variables. Let $T$ be a ``small base ring," a ring that is $p$-adically completed \'etale over the Tate algebra, $W(k)\langle t_1^{\pm 1}, \cdots , t_d^{\pm 1}\rangle$ (i.e. the ring of restricted Laurent series over $W(k)$), for some $d$. Let $r$ be an integer such that $0 \leq r \leq p-2.$ When the base ring is $W(k)$, Fontaine and Laffaille introduced in \cite{Fontaine} strongly divisible $W(k)$-lattices, which we now call Fontaine-Laffaille modules, to study $\mathbb{Z}_p$-lattices in crystalline representations of $G_K\coloneqq \mathrm{Gal}(\overline{K}/K)$ with Hodge-Tate weights in $[0, r].$ Denote this category by $\mathrm{Rep}_{\mathbb{Z}_p, [0,r]}^{\mathrm{cris}}(G_K)$. We first recall the definition of a Fontaine-Laffaille module in the classical setting:

\begin{definition}
\label{FLdefClassic}
 A \textit{finite free Fontaine-Laffaille module} over $W(k)$ is a triple $(M, \Fil^i(M), \varphi_i)$ where:
\begin{itemize}
\item $M$ is a finite free module over $W(k)$
\item $\{\Fil^i(M)\}_{i=0}^\infty$ is a decreasing filtration of $M$ such that $\Fil^{i+1} (M)$ is a direct summand of $\Fil^i(M)$ for all $i$, $\Fil^0(M) = M$, $\Fil^{r+1}(M) = 0.$ 
\item $\varphi_i$ are $\varphi$-semilinear maps $\Fil^i(M) \to M$ whose restriction to $\Fil^{i+1}(M)$ is $p\varphi_{i+1}.$  
\item $\sum_{i}\varphi_i (\Fil^iM)  = M.$
\end{itemize}
\end{definition}

Let $\mathrm{MF}^{\mathrm{ff}, [0,r]}(W(k))$ denote the category of finite free Fontaine-Laffaille modules over $W(k)$. These modules classify lattices within crystalline representations as in the following theorem:

\begin{theorem}[\cite{Fontaine}]
When $0 \leq r \leq p-2,$ the functor given by $$T_{\mathrm{cris}}(M) \coloneqq (\Fil^r(A_{\mathrm{cris}} \otimes_{W(k)} M))^{\varphi_r=1}$$ is an equivalence of categories $$\mathrm{MF}^{\mathrm{ff}, [0,r]}(W(k)) \leftrightarrow \mathrm{Rep}_{\mathbb{Z}_p, [0,r]}^{\mathrm{cris}}(G_K).$$ 
\end{theorem}

The work of Faltings in \cite{Faltings} extended this theory to more general base rings. Our goal is to extend the above theorem to the setting where the base ring is $R$ and to the setting of a small base ring $T$. \newline

Let us first recall the representation category in the relative case. Let $G_R$ denote $\mathrm{Gal}(\overline{R}[p^{-1}]/R[p^{-1}]) = \pi_1^{\mathrm{\acute{e}t}}(\mathrm{Spec}R\left[\frac{1}{p}\right], \eta)$ where $\eta$ is a fixed geometric point and  $\overline{R}$ is the union of finite $R$-subalgebras $R'$ of a fixed algebraic closure of $\mathrm{Frac}(R)$ such that $R'[p^{-1}]$ is \'etale over $R[p^{-1}].$ Let $\mathrm{Rep}_{\mathbb{Z}_p, [0,r]}^{\mathrm{cris}}(G_R)$ denote the category of crystalline $\mathbb{Z}_p$ representations of $G_R$ with Hodge-Tate weights in $[0,r]$ and we make the analogous definitions over the base ring $T$. \newline 

Using recent results of Du, Liu, Moon, and Shimizu in \cite{Fcrystals} and the category of finite, free Fontaine-Laffaille modules in the relative case first introduced by Faltings in \cite{Faltings} and more recently studied by \cite{LMP}, which we will denote $\mathrm{MF}^{\mathrm{ff}, [0,r]}_{\nabla}(R)$, we obtain the following result (see Definition \ref{RelFLDef} for the definition of Fontaine-Laffaille modules in the relative case, and see Section \ref{relativeperiodrings} for a brief review of the period rings and crystalline representations in the relative case): 

\begin{theorem}
Let $r$ be an integer satisfying $0 \leq r \leq p-2$ and let $\T$ denote either the base ring $R$ or $T$. The functor $$T_{\mathrm{cris}}: \mathrm{MF}^{\mathrm{ff}, [0,r]}_{\nabla}(\T)   \to \mathrm{Rep}_{\mathbb{Z}_p, [0,r]}^{\mathrm{cris}}(G_\T)$$ given by $$T_{\cris}(M) \coloneqq (\Fil^r(A_{\cris}(\T) \otimes_{\T} M))^{\varphi_r=1}$$ is an equivalence of categories.
\end{theorem} Showing this functor is fully faithful is routine, and the difficulty lies in showing essential surjectivity and specifically associating a Fontaine-Laffaille module to a given lattice inside a crystalline representation. \newline

The key input of \cite{Fcrystals} is the existence of a Kisin module $\GM$ over the ring $\GS \coloneqq \T\llbracket u \rrbracket$ associated to an object of $\mathrm{Rep}_{\mathbb{Z}_p, [0,r]}^{\mathrm{cris}}(G_
\T)$. The ring $\GS$ comes equipped with a Frobenius map $\varphi_{\GS}$ extending the Frobenius on $\T$. We can then define the associated Fontaine-Laffaille module as $\GM/u\GM \otimes_{\T, \varphi_{\GS}} \T$, but the Kisin module alone is not enough to obtain the full data of a Fontaine-Laffaille module over $\T$.\newline

In the classical case, Breuil in \cite{Breuil} studied $\mathbb{Z}_p$ lattices in crystalline representations of $G_K$ using strongly divisible $S$-modules. To obtain the full data of our Fontaine-Laffaille module, we define $S$ in the relative setting to be the $p$-adic completion of the divided power envelope of $\GS$ taken with respect to $(E(u))$ with $E(u) = u-p$ or $u+p$. The ring $S$ comes with a canonical divided power filtration and additional structure which will descend to our proposed Fontaine-Laffaille module. \newline

To show the main theorem holds over $R$, we will verify that the data properly descends through the use of two base change maps relating our $R$-module back to the classical theory. Let $k_g$ be the perfection of $\mathrm{Frac}(R/pR)$ and $R_g \coloneqq W(k_g)$, which gives rise to a flat embedding $R \hookrightarrow R_g$. Let $\overline{b}: R \to W(k)$ denote reduction moduo $(t_1, \cdots, t_d)$. A result of \cite{LMP} allows us to verify that $M$ is an object of $\mathrm{MF}^{\mathrm{ff}, [0,r]}_{\nabla}(R)$ by checking that $M \otimes_R R_g$ and $M \otimes_{W(k), \overline{b}} W(k)$ are Fontaine-Laffaille modules over $W(k_g)$ and $W(k)$ in the classical sense.\newline 

To show the main theorem holds over $T$, we obtain a result analogous to that of \cite{LMP} where we can check that $M$ is Fontaine-Laffaille at the $p$-adic completion of localizations at maximal ideals. This enables us to use a base change $T \to R$ as described in Section \ref{SmallBaseChange} and reduce to the setting where the base ring is $R$. \newline

Recently W\"urthen in \cite{Wurthen} proved this result in the setting of a smooth $p$-adic formal scheme over a mixed characteristic complete discrete valuation ring with perfect residue field. Our work uses different methods to focus on the case where the base ring is more limited, and it is routine to extend the results from a small base ring to accommodate the setting of a smooth, $p$-adic formal scheme.  \newline

The  work of Imai, Kato, and Youcis in  \cite{Kato},developed simultaneously, uses different methods to obtain a different fully faithful functor between lattices inside crystalline representations and Fontaine-Laffaille modules in a relative setting. Their functor has the advantage of being extendable beyond the Fontaine-Laffaille range, and it yields objects in a nicer category. Our functor has the advantage of being essentially surjective. They also place additional assumptions on the prismatic $F$-crystal, and their technique has not yet been applied to the case where the base ring is a power series ring. \newline

\textbf{Acknowledgements:} The author would like to thank his Ph.D advisor Tong Liu for suggesting the work in this paper, and for his extensive and invaluable guidance, comments, and conversation throughout the production of this paper and its earlier drafts. This paper includes and generalizes the content of the author's Ph.D. thesis at Purdue University when he was partially supported by the Ross Fellowship of Purdue University and a Summer Research Grant during the summer of 2021. 
\end{section}

\begin{section}{Preliminary Results on some Categories and Functors in the Relative Setting}
\begin{subsection}{Base Ring Conventions}
Much of integral $p$-adic Hodge Theory has been established in the relative case in far more generality than Fontaine-Laffaille theory. We present the most general setting we will consider here. \newline

Recall $k$ is a perfect field of characteristic $p > 2.$ Let $K^{tr}$ be a totally ramified extension of $K = \mathrm{Frac}(W(k))$ with ring of integers $\mathcal{O}_{K^{tr}}$. Fix $\pi$ a uniformizer of $K^{tr}$ and let $E(u) \in W(k)[u]$ denote the monic minimal polynomial of $\pi.$ Write $W(k)\langle T_1^{\pm 1}, \cdots , T_d^{\pm 1}\rangle$ for the $p$-adic completion of the Laurent polynomial ring $W(k)\left[ T_1^{\pm 1}, \cdots , T_d^{\pm 1}\right]$. \newline
\begin{convention}
\label{generalbasering}
In order to reference results with maximal generality, we will say $\mathcal{R}$ is a \emph{general base ring} if $\mathcal{R}$ is a $p$-adically complete $\mathcal{O}_{K^{tr}}$-algebra which is of the form $\mathcal{R} = \mathcal{O}_{K^{tr}} \otimes_{W(k)} \mathcal{R}_0$ where $\mathcal{R}_0$ is an integral domain obtained from $W(k)\langle T_1^{\pm 1}, \cdots , T_d^{\pm 1}\rangle$ by a finite number of iterations of the operations:
\begin{itemize}
\item $p$-adic completion of an \'etale extension
\item $p$-adic completion of a localization;
\item completion with respect to an ideal containing $p$.
\end{itemize}
\end{convention}

Our new results will hold for a subset of these general base rings which we describe next: 

\begin{convention}
\label{small convention}
We will say a general base ring $T$ is \emph{small} if it is $p$-adically completed \'etale over $W(k)\langle T_1^{\pm 1}, \cdots , T_d^{\pm 1}\rangle$ for some $d$. 
\end{convention}

The key strategy for showing our results over small base rings will be to reduce to the case of the power series ring $R = W(k)\llbracket t_1, \cdots , t_d\rrbracket$ due to the local structure of small base rings described in the following lemma:

\begin{lemma}
\label{TateLocalStructure}
Let $T$ be a small base ring that is $p$-adically compaeted \'etale over $W(k)\langle T_1^{\pm 1}, \cdots , T_d^{\pm 1}\rangle$ and let $\mathfrak{m}$ be a maximal ideal (note that $p \in \mathfrak{m}$). Then $\widehat{T_{\mathfrak{m}}}$, the $p$-adic completion of the localization of $T$ at $\mathfrak{m}$, is isomorphic to $W(k)\llbracket t_1, \cdots , t_d \rrbracket$.
\begin{proof}
Let $\mathfrak{m}$ be a maximal ideal of $T$, and let $\Tatem$ denote the $p$-adic completion of the localization of $T$ at $\mathfrak{m}.$ The ring $T$ is a regular Noetherian ring, and thus $\Tatem \cong W(k)\llbracket X_1, \cdots , X_d\rrbracket \cong R$ by Cohen's Structure Theorem.
\end{proof}
\end{lemma}

We will try to reserve the letter $R$ in usual font for referring to such a power series ring, the letter $T$ in usual font for a small base ring, and the stylized $\mathcal{R}$ for a general base ring. We will use the stylized $\mathcal{T}$ to concisely state results that hold for a small base ring or the power series ring $R$. \newline 

When we refer to the ``classical" theory we will be referring to the setting where the base ring is $K$: the historically typical setting of $p$-adic Hodge Theory. \newline

We will now describe the Frobenius structure and the category of Galois representations for a general base ring. We will let $\varphi_{\mathcal{R}}$ denote the lift of Frobenius on $\mathcal{R}/W(k)$ uniquely determined by $\varphi(T_i) = T_i^p$. Let $\overline{\mathcal{R}}$ denote the union of finite $\mathcal{R}$-subalgebras $\mathcal{R}'$ of a fixed algebraic closure of $\mathrm{Frac}(\mathcal{R})$ such that $\mathcal{R}'[p^{-1}]$ is \'etale over $\mathcal{R}[p^{-1}]$. Set $G_{\mathcal{R}} = \mathrm{Gal}(\overline{\mathcal{R}}/\mathcal{R}[p^{-1}])$. Let $\mathrm{Rep}_{\mathbb{Q}_p}(G_{\mathcal{R}})$ denote the category of finite dimensional $\mathbb{Q}_p$ vector spaces with continuous $G_{\mathcal{R}}$ action, and let $\mathrm{Rep}_{\mathbb{Z}_p}^{\mathrm{ff}}(G_{\mathcal{R}})$ denote the full category of finite free $\mathbb{Z}_p$ modules equipped with a continuous action of $\mathcal{G}_\mathcal{R}$. \newline 

\begin{remark}
This describes the Frobenius structure for a general base ring, but we will actually take $\varphi_R(t_i) = (t_i+1)^p -1$ for our base ring $R$ in order to use the results of \cite{LMP} as stated. This is not significant, as we will show the relevant category of Fontaine-Laffaille modules defined below is independent of the choice of Frobenius. \newline
\end{remark}
\end{subsection}
\begin{subsection}{Important Base Change Maps}
\label{BaseChangeSection}
Our main technique for developing the Fontaine-Laffaille Theory for new base rings will be to base change to rings where the Fontaine-Laffaille Theory is already known. We describe the important base change maps here.

\begin{subsubsection}{Shilov Point Base Changes}
\label{ShilovPointBaseChange}
Let $k_g$ be the perfection of $\mathrm{Frac}(\mathcal{R}/p\mathcal{R})$ and let $\mathcal{R}_g \coloneqq W(k_g)$. We have a $\varphi$-compatiable, flat embedding $\mathcal{R} \hookrightarrow \mathcal{R}_g$. Set $\GS_{\mathcal{R},g} \coloneqq W(k_g)\llbracket u\rrbracket$,  and let $S_{\mathcal{R},g}$  be the $p$-adically completed divided power envelope of $\GS_{\mathcal{R},g}$ with respect to $E(u)$ which comes equipped with a $\varphi$ structure and PD-filtration. \newline
 
As the Shilov point above forgets the data of the connection $\nabla$, we will also need to consider a second Shilov point with imperfect residue field. Let $\mathcal{R}_L$ denote the $p$-adic completion of the localization $\mathcal{R}_{(p)}$. Then $\mathcal{R}_L$ is a complete discrete valuation ring with imperfect residue field. We define $\GS_L = \mathcal{R}_L \llbracket u \rrbracket$ and $S_L$ to be the $p$-adically completed divided power envelope of $\GS_L$ with respect to $(E(u)).$ 
\end{subsubsection}
\begin{subsubsection}{Closed Fiber Base Change of a Power Series Ring}
\label{ClosedPointBaseChange}
In the case of our base ring $R$, we also have a natural projection $\pi: R \to W(k)$ given by $t_i \mapsto 0$ for every $i$. Showing that the Fontaine-Laffaille Theory holds for $R$ will amount to checking it is compatible with the Shilov point base change and the closed fiber base change.
\end{subsubsection}
\begin{subsubsection}{Small Base Ring Base Changes}
\label{SmallBaseChange}
We will now discuss a base change map from a small base ring $T$ to $R$ suggested by Lemma \ref{TateLocalStructure}. Let $\mathfrak{m} = (p, f_1, \cdots , f_d)$ be a maximal ideal of $T$ (recall $T$ is regular), and let $\Tatem$ denote the $p$-adic completion of the localization of $T$ at $\mathfrak{m}$ which is isomorphic to $\Tatem \cong W(k)\llbracket X_1, \cdots , X_d\rrbracket \cong R$ by Lemma \ref{TateLocalStructure}. \newline

This gives us a flat base change ring homomorphism $b_{\mathfrak{m}}: T \to \Tatem \cong R$. In order to make this map Frobenius equivariant, though, we must consider $\varphi_R$ to be the lift of the Frobenius on $W(k)$ which acts on the $X_i$ as $\varphi_R(X_1) = b_\mathfrak{m}(\varphi_T(f_i))$. For example, if $\mathfrak{m} = (p, t_1-1, \cdots, t_d - 1)$, then we must take $\varphi_R(X_i) = (X_i+1)^p-1.$ \newline 

By Lemma \ref{FrobEquiv} below, there is a canonical equivalence between the categories of finite, free Fontaine-Laffaille modules over $R$ with different Frobenii, so this choice will not impact our results.  
\end{subsubsection}
\end{subsection}
\begin{subsection}{Review of relative period rings and crystalline representations}
\label{relativeperiodrings}
In this section we briefly review the crystalline period rings and the functor $D_{\cris}$ in the relative case. For full details of their constructions, see \cite[Chapter 6]{BrinonPeriodRings} for a treatment of the period rings and \cite[Chapter 8]{BrinonPeriodRings} for a treatment of relevant functors. Let $\mathcal{R}$ be a general base ring as in Convention \ref{generalbasering}. Let $\widehat{\overline{\mathcal{R}}}$ be the $p$-adic completion of $\overline{\mathcal{R}}$. Set $\overline{\GR}^\flat = \displaystyle \varprojlim_{\varphi} \overline{\GR}/p\overline{\GR}$. Then we define $A_{\mathrm{inf}}(\GR) = W(\overline{\GR}^\flat)$ and write $[\pi^\flat]$ for the Teichm\"uler lift of $\pi^\flat$ where $\pi^\flat = (\pi^{1/p^n})_{n \geq 0}$ is a compatible sequence of $p$-power roots of $\pi$. Let $\theta: \Ainf(\GR) \to \widehat{\overline{\GR}}$ be the unique surjective $W(k)$-algebra homomorphism lifting the first projection.  \newline

Let $\Acris(\GR)$ be the $p$-adic completion of the divided power envelope of $\Ainf(\GR)$ with respect to the kernel of $\theta$. Let $\epsilon_n \in \overline{\GR}$ denote a (non-trivial) sequence of compatible $p$-power roots of unity (i.e. such that $\epsilon_0 = 1), \epsilon_1 \neq 1, \epsilon_{n} = \epsilon_{n+1}^p$. Let $\epsilon = (\epsilon_n)_n \in \overline{\GR}^\flat$ and $t = \log[\epsilon] \in \Acris(\GR).$ Let $\Bcris(\GR) \coloneqq \Acris(\GR)[p^{-1}, t^{-1}].$  \newline

Let $\theta_0$ be the extension of $\theta$ to $\GR_0 \otimes_{W(k)} \Ainf(\GR) \to \widehat{\overline{\GR}}$ and set $\OAcris(\GR)$ to be the $p$-adic completion of the divided power envelope of $\GR_0 \otimes_{W(k)} \Ainf(\GR)$ with respect to the kernel of $\theta_0$. Let $\OBcris(\GR) \coloneqq \OAcris(\GR)[p^{-1}, t^{-1}]$.\newline

For $V \in \mathrm{Rep}_{\mathbb{Q}_p}(G_{\mathcal{R}})$, we have a functor $$D_{\cris}(V) \coloneqq (\mathrm{OB}_{\cris}(\mathcal{R}) \otimes_{\mathbb{Q}_p} V)^{G_{\mathcal{R}}}$$ which is a finite projetive $\mathcal{R}[p^{-1}]$ module of rank at most $\mathrm{dim}_{\mathbb{Q}_p}V$ equipped with a $\varphi$ and an integrable connection $\nabla$  induced by that on $\mathrm{OB}_{\cris}(\mathcal{R}).$ We say that $V$ is crystalline if the natural map $$\mathrm{OB}_{\cris}(\mathcal{R}) \otimes_{\mathcal{R}[p^{-1}]} D_{\cris}(V) \to \mathrm{OB_\cris}(\mathcal{R}) \otimes_{\mathbb{Q}_p} V$$ is an isomorphism. An object $N \in \mathrm{Rep}_{\mathbb{Z}_p}^{\mathrm{ff}}(G_{\mathcal{R}})$ is said to be \emph{crystalline} if $N \otimes_{\mathbb{Z}_p} \mathbb{Q}_p$ is crystalline. There is also a contravariant version of $D_\cris$ defined to be $$D_\cris^\vee(V) \coloneqq \Hom_{G_{\mathcal{R}}}(V, \mathrm{OB}_\cris(\mathcal{R}))$$ dual to the above covariant version. 

\end{subsection}
\begin{subsection}{Finite free Fontaine-Laffaille modules}
Here we introduce the various categories of Fontaine-Laffaille modules in the relative case, following the work of \cite{Faltings} and \cite{LMP}. Here we will need to distinguish between the case of the base ring $R$ and a small base ring $T$ for the statement of some theorems. 

\begin{definition}
\label{RelFLBigDef}
Let $\mathrm{MF}_{\mathrm{big}}^{\mathrm{f}}(R)$ be the category whose objects consist of a an $R$-module $M$, a sequence of $R$-modules $\mathrm{F}^i(M)$, and sequences of $R$-linear maps $\iota_i: \mathrm{F}^i(M) \to \mathrm{F}^{i-1}(M), \pi_i: \mathrm{F}^i(M) \to M,$ and $R$-semi-linear maps $\varphi_i: \mathrm{F}^i(M) \to M$, satisfying the following conditions: \newline
(1) The composition $\mathrm{F}^i(M) \xrightarrow{\iota_i}{} \mathrm{F}^{i-1}(M) \xrightarrow{\pi_{i-1}}{} M $ is the map $\pi_i: \mathrm{F}^i(M) \to M$\newline
(2) The map $\pi_i : \mathrm{F}^i(M) \to M$ is an isomorphism for $i \ll 0.$ \newline
(3) The composition of $\varphi_{i-1}$ with $\iota_i: \mathrm{F}^i(M) \to \mathrm{F}^{i-1}(M)$ is $p\varphi_i.$ \newline
Morphism between these objects are compatible collections of $R$-linear maps between $M$'s and $\mathrm{F}^i(M)$'s.
\end{definition}

We write $\mathrm{MF}_{\nabla, \mathrm{big}}^{\mathrm{f}}(R)$ when we consider the category of such modules equipped with an integrable connection $\nabla: M \to M \otimes_{R} \widehat{\Omega}_{R}$ for which the $\varphi_i$ are parallel and that satisfies Griffiths transversality. Specifically, when the following conditions hold:
\begin{itemize}
    \item Griffiths-transversality holds. i.e. the following diagram commutes: 
    \adjustbox{scale=1,center}{%
  \begin{tikzcd}[scale=1, column sep=large]
    \mathrm{F}^i(M) \arrow{d}{\nabla_i} \arrow{r}{\iota_i} & \mathrm{F}^{i-1}M \arrow{d}{\nabla_{i-1}}\\
    \mathrm{F}^{i-1}M\otimes_R \widehat{\Omega}_R \arrow{r}{\iota_{i-1} \otimes_R \widehat{\Omega}_R} & \mathrm{F}^{i-2}(M) \otimes_R \widehat{\Omega}_R
  \end{tikzcd}
}

    Note that this simplifies to $\nabla (\mathrm{F}^i(M)) \subset \mathrm{F}^{i-1}(M) \otimes_R \widehat{\Omega}_R$ when the $\iota_i$ are injective. 
    \item The $\varphi_i$ are parallel: $\nabla \circ \varphi_i = (\varphi_{i-1} \otimes_R d\varphi_1) \circ \nabla$ as a map from $\mathrm{F}^i(M)$ to $M \otimes_R \widehat{\Omega}_R$ where $d\varphi_1: \widehat{\Omega}_R \otimes_{R, \varphi} R \to \widehat{\Omega}_R$ is $d\varphi /p$.
\end{itemize}

We write $\mathrm{MF}_{\nabla, \mathrm{big}}^{\mathrm{f}, [a,b]}(R)$ when $\mathrm{F}^a(M) = M$ and $\mathrm{F}^{b+1}(M) = \{0\}.$ \newline 

We will replace ``f" with ``tor" when working in the setting of $p$-power torsion modules as in \cite{LMP}. We define these categories analgously when working over small base rings.

\begin{definition}
\label{RelFLDef}
Let $\mathrm{MF}^{\mathrm{ff}}(R)$ denote the full subcategory of $\mathrm{MF}^{\mathrm{f}}(R)$ such that $M$ and $\mathrm{F}^i(M)$ are finitely generated and free, each $\mathrm{F}^{i+1}(M)$ is a direct summand of $\mathrm{F}^{i}(M)$, $\mathrm{F}^i(M) = \{0\}$ for $i \gg 0$, and $\sum_{i} \varphi_i(\mathrm{F}^i(M))$ generate $M.$  We call an object in this category a ``finite free Fontaine-Laffaille module (over $R$)" or sometimes just a ``Fontaine-Laffaille module (over $R$)" for convenience. The tags $\nabla$, $[a, b]$, and ``tor" are used in the same way as for the big category, and we define the categories analagously for small base rings.
\end{definition} 

We next recall a main theorem of \cite{LMP} which we will refine to our finite free setting over the base ring $R$, and we will develop an analogue for small base rings. Recall the Shilov point $R_g$ from Section \ref{ShilovPointBaseChange} and that here is a $\varphi$-compatible, flat embedding $b_g: R \hookrightarrow R_g$. Recall from Section \ref{ClosedPointBaseChange} that there is a natural projection $\overline{b}: R \to W(k)$ determined by $t_i \mapsto 0$ for each $i$. A main theorem of \cite{LMP} is as follows:  \begin{theorem}{\cite[Proposition 2.2.5]{LMP}}
\label{genericclosedtorsion}
\newline Suppose $M \in \mathrm{MF}_{\mathrm{big}}^{\mathrm{tor}, [0, r]}(R)$ such that each $\mathrm{\Fil^i(M)}$ is finite as an $R$-module. Then $M$ is in $\mathrm{MF}^{\mathrm{ff}}(R)$ if and only if: \newline
(1) Both $M_0 \coloneqq M \otimes_{R, \overline{b}} W(k)$ and $M_g \coloneqq M \otimes_{R} R_g$ are objects in $\mathrm{MF}^{\mathrm{tor}, [0,r]}(W(k_g))$ and $\mathrm{MF}^{\mathrm{tor}, [0,r]}(R_g),$ respectively. \newline 
(2) $M_g \cong M_0 \otimes_{W(k)} R_g$ as $R_g$-modules.
\end{theorem} 

Lt $\mathrm{MF}_{\mathrm{big}}^{\mathrm{ff}, [0,r]}(R)$ denote the full subcategory of $\mathrm{MF}_{\mathrm{big}}^{\mathrm{f}, [0,r]}(R)$ where $M$ and each $\Fil^i(M)$ are finite, free $R$-modules. Then we have:

\begin{proposition}
\label{genericclosed}
Theorem \ref{genericclosedtorsion}, holds if each ``$\mathrm{tor}$" is replaced with ``$\mathrm{ff}.$"
\begin{proof}
If $M \in \mathrm{MF}_{\mathrm{big}}^{\mathrm{ff}, [0,r]}(R)$, by tensoring with $W(k)$ and $R_g$ and noting that $M$ and $\Fil^i(M)$ are finite free $R$-modules, it is clear that (1) and (2) hold. \newline

We can show the reverse direction by reducing to the torsion case. Let $n$ be a positive integer. It is clear that $M/p^nM \otimes_{R, \overline{b}} W(k) \cong M \otimes_{R} R/p^nR \otimes_{R, \overline{b}} W(k) \cong M_0/p^nM_0$ and $M/p^nM \otimes_{R} R_g \cong M_g/p^nM_g$ are elements of $\mathrm{MF}^{\mathrm{tor}, [0,r]}(W(k))$ and $\mathrm{MF}^{\mathrm{tor}, [0,r]}(R_g)$, respectively. \newline 

By tensoring $M_g \cong M_0 \otimes_{W(k)} R_g$ with $R/p^nR$, we also see that $M_0/p^nM_0$ and $M_g/p^nM_g$ are of the same type. Then by Theorem \ref{genericclosedtorsion}, we conclude that $M/p^nM$ is an element of $\mathrm{MF}^{\mathrm{tor}, [0,r]}(R).$ Since $M = \varprojlim M/p^nM$ and $\mathrm{F}^i(M) = \varprojlim \mathrm{F}^i(M)/p^n\mathrm{F}^i(M)$ and each is free, we deduce that each is finite free and we have the direct summand condition. Similarly, as $\varphi_i(\mathrm{F}^i(M/p^nM))$ generates $M/p^nM$, the same condition lifts to $M$.
\end{proof}
\end{proposition} 

We will now build to a theorem analogous to Theorem \ref{genericclosed} but which holds for small base rings. First we will show that the category $\mathrm{MF}_{\nabla}^{\mathrm{ff}}(R)$ is independent of the lift of Frobenius. 
\begin{lemma}[\cite{Faltings}]
\label{FrobEquiv}
Let $a,b$ be such that $0 \leq b-a \leq p-1$. Let $\varphi, \psi$ be two lifts of the natural Frobenius on $W(k)$ to $R$ which coincide modulo $p$. We write $R_{\varphi}$ (resp. $R_{\psi}$) when we are considering $R$ with the Frobenius lift $\varphi$ (resp. $\psi$). Then there is an equivalence between the corresponding categories $\mathrm{MF}_{\nabla}^{\mathrm{ff},[a,b]}(R_{\varphi})$ and  $\mathrm{MF}_{\nabla}^{\mathrm{ff},[a,b]}(R_{\psi})$, and up to canonical isomorphism, $\mathrm{MF}_{\nabla}^{\mathrm{ff}, [a,b]}(R)$ is independent of the lift of Frobenius.  

\begin{proof}
We show the result for the torsion categories, and the result for the finite free categories follows via projective limits. \newline 

The proof of \cite{Faltings} applies almost identically in the torsion case, but we do not have an \'etale map from $W(k)[t_1, \cdots , t_d] \to R$, but we still have that $\{\partial_i\}_{i=1}^d \coloneqq \{\partial/\partial X_i\}_{i=1}^d$ gives the dual basis of $R$-derivations. From there we proceed identically to Faltings. By shifting if necessary, we can assume that $a = 0$, $b = p-1.$ The $\partial/\partial X_i$ act on $M$ via $\nabla.$ Given a multi-index $I = (i_1, \cdots, i_d)$, we get an endomorphism $\nabla(\partial)^I$ of $M$. Let $(\varphi(X)-\psi(X))^I$ denote $\prod_{j=1}^d(\varphi(X_j)-\psi(X_j))^{i_j}$, and $|I| = \sum_{j=1}^d i_j$, and $I! = \prod_{j=1}^d i_j!$. Comparing divisibility by $p$, we can see this gives a well defined element. \newline 

Note that an object $M \in \mathrm{MF}_{\mathrm{big}}^{\mathrm{tor}, [a,b]}(R)$ is in $\mathrm{MF}^{\mathrm{tor}, [a,b]}(R)$ if and only if the Frobenius $\phi$ on $M$ induces an isomorphism $\widetilde{M} \otimes_{R, \varphi} \to M$ where $\widetilde M$ is the right exact functor defined explicitly as the cokernel of the map $$\theta_M: \bigoplus_{i=a+1}^b \mathrm{F}^i(M) \to \bigoplus_{i=a}^b \mathrm{F}^i(M),$$ where $\theta_M$ is given by $$\theta_M((x_{a+1}, \cdots, x_b)) = (\iota_{a+1}(x_{a+1}), -px_{a+1}+\iota_{a+2}(x_{a+2}), \cdots, -px_{b-1} + \iota_{b}(x_b) , -px_b).$$  Equivalently, $\widetilde{M}$ can be defined as the colimit of the following diagram:
$$\cdots \rightarrow \mathrm{F}^{i+1}(M) \leftarrow \mathrm{F}^{i+1}(M) \rightarrow \mathrm{F}^i(M) \leftarrow \mathrm{F}^i(M) \rightarrow \mathrm{F}^{i-1}(M) \leftarrow \cdots$$ where the right arrows denote the $\iota$ maps and the left arrows denote multiplication by $p$.  \newline 

Now we define a map $\alpha: \widetilde{M} \otimes_{R, \varphi} R \to \widetilde{M} \otimes_{R, \psi} R$. Let $m \in F^i(M)$, which defines an element of $\widetilde{M}$, and it suffices to show $\alpha$ is an isomorphism. We define: $$\alpha(m \otimes 1) = \sum_{I} \nabla(\partial)^I(m) \otimes \frac{(\varphi(X) - \psi(X))^I}{I! \cdot p^{\min(|I|, i)}}.$$ Noting that $p^{|I|}$ divides $(\varphi(X)-\psi(x))^I$, we can confirm that $\frac{(\varphi(X) - \psi(X))^I}{I! \cdot p^{\min(|I|, i)}}$ gives a well-defined element of $R$ and that this sum converges to $0$ in the $p$-adic topology. If $|I| \leq i$, the fraction is obviously well-defined as it forces $|I| \leq p-1$ and $v_p(I!) = 0.$ If $|I| > i,$ this can be done by writing each $i_j$ in its base $p$ representation as $i_j = \sum_{k=0}^\infty a_{j,k}p^k$ and using the well-known identity $$v_p(i_j!) = \frac{1}{p-1}\sum_{k=0}^\infty a_{j,k}(p^k-1).$$ Then the $p$-adic valuation of the fraction becomes at least 
\begin{align*}&|I| - \mathrm{min}(|I|, i) - \sum_{j=1}^d \frac{1}{p-1}\sum_{k=0}^\infty a_{j,k}(p^k-1)\\
&= -i +\sum_{j=1}^d\sum_{k=1}^\infty a_{j,k}\left[p^k - \frac{p^k-1}{p-1}\right] \\
&= -i +\sum_{j=1}^d\sum_{k=1}^\infty a_{j,k}\left[ \frac{p^k(p-2) +1}{p-1}\right]
\end{align*}
We can then see that for $p > 2$, $p^k(p-2)+1 > (p-1)^k$. Since $i \leq p-1 < |I|$, the $p$-adic valuation is nonnegative and thus the original fraction is well-defined, and we can see that $|I| - v_p(I!)$ grows without bound as $|I|$ does. \newline 

\color{black}As in Faltings, Taylor's formula $\varphi(r) = \sum_{I} \psi(\nabla(\partial)^I(r)) \otimes (\varphi(X)-\psi(X))^I/I!$ shows us that $\alpha$ gives a well-defined map. The $\alpha$'s satisfy transitivity for three different Frobenius lifts by the binomial formula. Applying this to $\varphi, \psi, \varphi$ gives us that the $\alpha$ are isomorphisms. It is easy to verify that $\alpha$ is parallel for the connections. 
\end{proof}
\end{lemma}

We are now ready to state the analogue of Theorem \ref{genericclosed} for a small base ring $T$: 
\begin{theorem}
\label{TateLMPTheorem}
    Suppose $M$ is in $\mathrm{MF}_{\mathrm{big}}^{\mathrm{f}}(T)$ such that each $\mathrm{F}^i(M)$ is finite as a $T$-module. Then $M$ is in $\mathrm{MF}^{\mathrm{ff}}(T)$ if and only if $(M_{\mathfrak{m}}, \mathrm{F}^i(M_{\mathfrak{m}})) \coloneqq (M \otimes_{b_{\mathfrak{m}}, T} \Tatem , \mathrm{F}^i(M) \otimes_{b_{\mathfrak{m}},T} \Tatem) \in \mathrm{MF}^{\mathrm{ff}}(\widehat{T_{\mathfrak{m}}})$ for all maximal ideals $\mathfrak{m}$ of $T$.  

\begin{proof}
Again we show the result in the torsion case and the result in the finite free case follows from taking projective limits. \newline

For the forward direction, it is clear that $M_{\mathfrak{m}}, F^i(M_{\mathfrak{m}})$ are finitely generated $p$-power torsion $\Tatem$-modules with $\mathrm{F}^i(M_{\mathfrak{m}}) = \{0\}$ for $i \gg 0$ since $M$ and $\mathrm{F}^i(M)$ have these properties over $T$. The axioms of $\mathrm{MF}_{\mathrm{big}}(\Tatem)$ are verified directly by properties of the tensor product and the fact that $b_{\mathfrak{m}}$ is $\varphi$-equivariant with our choice of Frobenius. Recalling that $\widetilde{M}$ is a cokernel and again using that $b_{\mathfrak{m}}$ is $\varphi$-equivariant, we obtain an isomorphism $\widetilde{M \otimes_{b_{\mathfrak{m}}, T} \Tatem} \cong \widetilde{M} \otimes_{b_{\mathfrak{m}}, T} \Tatem.$ Thus $\varphi_{M_{\mathfrak{m}}}: \widetilde{M_{\mathfrak{m}}} \otimes_{\Tatem, \varphi_{\mathfrak{m}}} \Tatem$ induces an isomorphism, and $M_{\mathfrak{m}}$ is an object in $\mathrm{MF}(\Tatem).$ \newline 

For the reverse direction, we are assuming that each $\mathrm{F}^i(M)$ is a finite $T$ module, so we need only check that $\varphi_M : \widetilde{M} \otimes_{T, \varphi} T \to M$ is an isomorphism. It suffices to check that this is an isomorphism locally at maximal ideals. Since $\widehat{T_\mathfrak{m}}$ is a Noetherian local ring and $(p)$ is an ideal contained in $\mathfrak{m}$, the ring map $T_\mathfrak{m} \to \widehat{T_\mathfrak{m}}$ is faithfully flat. 
Thus it suffices to check that $\widehat{T_\mathfrak{m}} \otimes \varphi_M$ is an isomorphism for every maximal ideal $\mathfrak{m}$, which is true by assumption since, as in the previous paragraph, $\widetilde{M_{\mathfrak{m}}} \cong \widetilde{M \otimes_{b_{\mathfrak{m}}, T} \Tatem} \cong \widetilde{M} \otimes_{b_{\mathfrak{m}}, T} \Tatem.$
\end{proof}
\end{theorem}
\end{subsection}
\begin{subsection}{\'Etale $(\varphi_{\mathcal{O}_{\mathcal{E}}}, \mathcal{O}_{\mathcal{E}})$-modules}
In this section and the next we will work over a general base ring $\GR$ as in Convetion \ref{generalbasering}. Here we review the theory of \'etale $(\varphi_{\mathcal{O}_{\mathcal{E}}}, \mathcal{O}_{\mathcal{E}})$-modules in the relative case. \newline

Let $\GS_\mathcal{R} \coloneqq \mathcal{R}\llbracket u\rrbracket$ with Frobenius $\varphi_{\GS_{\mathcal{R}}}$ obtained by setting $\varphi_{\GS_{\mathcal{R}}}(u) = u^p.$ Let $\mathcal{O}_{\mathcal{E}, \mathcal{R}} = \etalering$ be the $p$-adic completion of $\GS_{\mathcal{R}}[u^{-1}]$ with Frobenius $\varphi_{\mathcal{O}_{\mathcal{E}}}$ extending that on $\GS_{\mathcal{R}}.$ \newline 

\begin{definition}
An \'etale ($\varphi_{\etalering}$, $\etalering$)-module is a pair $(\mathcal{M}, \varphi_{\mathcal{M}})$ where $\mathcal{M}$ is a finitely generated $\etalering$-module and $\varphi_{\mathcal{M}}: \mathcal{M} \to \mathcal{M}$ is a $\varphi$-semilinear endomorphism such that the linearization $1 \otimes \varphi_{\mathcal{M}}: \varphi^*\mathcal{M} \to \mathcal{M}$ (i.e. the map $c \otimes m \mapsto c\varphi_{\mathcal{M}}(m)$ for $c \otimes m \in \etalering \otimes_{\varphi_{\etalering}, \etalering} \mathcal{M}$) is an isomorphism. 
\end{definition}

Let $\mathrm{Mod}^{pr}_{\mathcal{O}_{\mathcal{E}}}$ denote the category of projective \'etale ($\varphi_{\etalering}$, $\etalering$)-modules whose morphisms are $\varphi_{\etalering}$-compatible $\etalering$-linear maps and $\mathrm{Mod}^{ff}_{\mathcal{O}_{\mathcal{E}}}$ the category with the projective condition replaced by finite free. We now recall how \'etale ($\varphi_{\etalering}$, $\etalering$)-modules relate to Galois representations. \newline 

Recall that $\pi_n \in \overline{K^{tr}}$ is chosen compatibly so that $\pi_0 = \pi$ and $\pi^p_{n+1} = \pi_n$. Set $K_{\infty}$ to be the $p$-adic completion of $\cup_{n=0}^\infty K^{tr}(\pi_n)$ and $K_{\infty}^\flat$ its tilt. Let $E_{\mathcal{R}_{\infty}}^+ = \GS/p\GS$ and $\tilde{E}_{\mathcal{R}_{\infty}}^+$ denote the $u$-adic completion of $\varinjlim_{\varphi} E_{R_\infty}^+$. Write $\tilde{\mathcal{R}}_\infty$ for $W(\tilde{E}_{\mathcal{R}_{\infty}}^+) \otimes_{W(K_\infty^\flat)} \mathcal{O}_{K_{\infty}}$.

\begin{remark}
While we have described $\tilde{\mathcal{R}}_{\infty}$ in full generality for completeness, for our base ring $R$ we will only need to use $\tilde{R}_{\infty}$ which has a much more explicit description as $$\tilde{R}_{\infty} = \bigcup_{n \geq 1, 1 \leq i \leq d} R\left(\pi_n, \sqrt[p^n]{1+t_i}\right)$$ for a fixed choice of compatibly chosen $p^n$th roots of $1+ t_i$.
\end{remark}

Then we have the following relationship: 

\begin{proposition}{\cite[Prop. 7.7]{Kim}, \cite[Prop 2.16]{Fcrystals}}
\label{latticeetalephi}
There is a functor $\mathcal{M}$ from the category $\mathrm{Rep}_{\mathbb{Z}_p}^{\mathrm{ff}}(G_{\tilde{\mathcal{R}}_{\infty}})$ of finite free $\mathbb{Z}_p$-modules with continuous $G_{\tilde{\mathcal{R}}_{\infty}}$ action to the category $\mathrm{Mod}^{pr}_{\mathcal{O}_{\mathcal{E}}}$ which is an exact equivalence of categories. The inverse of $\mathcal{M}$ is given by $$T_{\etalering}(\mathcal{M}) \coloneqq (\widehat{\etalering}^{\mathrm{ur}} \otimes_{\etalering} \mathcal{M})^{\varphi=1}.$$ This functor behaves well with respect to base change in the following sense: Let $\mathcal{R}'$ be another base ring equipped with Frobenius satisfying the same conditions as $\mathcal{R}$ with a $\varphi$-equivariant map $\mathcal{R}_0 \rightarrow \mathcal{R}_0'$. If $T \in \mathrm{Rep}_{\mathbb{Z}_p}^{\mathrm{ff}}(G_{\tilde{R}_{\infty}})$, then $T$ can be considered as a $G_{\tilde{\mathcal{R}}_{\infty}'}$-representation via the map $G_{\tilde{\mathcal{R}}_\infty'} \rightarrow G_{\tilde{\mathcal{R}}_\infty}$ and we have the isomorphism $$\mathcal{O}_{\mathcal{E}, \mathcal{R'}} \otimes_{\etalering} \mathcal{M}(T) \cong \mathcal{M}_{\tilde{\mathcal{R}}'}(T)$$ as \'etale ($\varphi_{\mathcal{O}_{\mathcal{E}, \mathcal{R}'}}$, $\mathcal{O}_{\mathcal{E}, \mathcal{R}'}$)-modules. 
\end{proposition}
\end{subsection}
\begin{subsection}{Kisin Modules}
\label{KisinModulesSection}
We next review Kisin modules in the relative case and the main result of \cite{Fcrystals} which provides the key input to the proof of our main theorem. \newline

First we recall the the ring $\GST$ considered in \cite{Fcrystals}. Let $\mathcal{R}_\Prism$ denote the absolute prismatic site of $\mathcal{R}$. In full generality, $\GS$ is the ring $\mathcal{R}_0\llbracket u\rrbracket$ and $E$ is an choice of minimal Eisenstein polynomial of $\pi$ (in our setting, $E(u) = u-p$ or $u+p$). We write $(\GS, (E))$ for the Breuil-Kisin prism. Then $(\GST, (E))$ is defined to be the self-product of $(\GS, (E))$ in $\mathcal{R}_\Prism$ and that it satisfies a universal property as follows: \newline

If $(B, I) \in \mathcal{R}_\Prism$ and if we are given maps $f_1, f_2: (\GS, (E)) \to (B, I)$ such that the maps $\mathcal{R} \to B/I$ induced by $f_1, f_2$ agree, then there is a map $(\GST, (E)) \to (B, I)$ uniquely determined by $f_1, f_2.$ We will write $p_1, p_2$ for the maps $(\GS, (E)) \to (\GST, (E))$. We will also utilize the triple-self product of $(\GS, (E))$ denoted $(\GS^{(2)}, (E))$. \newline

The details of the construction of $(\GST, (E))$ and a justification that this self-product exists can be found in \cite[Example 3.4]{Fcrystals}. We now introduce the category of free Kisin modules with descent data. \newline

\begin{definition}[\cite{Fcrystals} Definition 3.24]
\label{dddef}
Let $\mathrm{DD}_{\GS}$ denote the category consisting of triples $(\GM, \varphi_{\GM}, f)$ where
\begin{itemize}
    \item $\GM$ is a finite $\GS$-module that is projective away from $(p,E)$ and saturated.
    \item $\varphi_{\GM}: \GM \to \GM$ is a $\varphi$-semi-linear endomorphism such that $(\GM, \varphi_{\GM})$ has finite $E$-height.
    \item $f: \GST \otimes_{p_1, \GS} \GM \to \GST \otimes_{p_2, \GS} \GM$ is an isomorphism of $\GST$-modules that is compatible with Frobenii and satisfies the cocycle condition over $\GS^{(2)}.$
    \end{itemize}
An object of this category is called an \textbf{integral Kisin descent datum}. If we replace the third piece of data with a map $f: \GST[p^{-1}] \otimes_{p_1, \GS} \GM \to \GST[p^{-1}] \otimes_{p_2, \GS} \GM$ that is an isomorphism of $\GST[p^{-1}]$-modules that is compatible with Frobenii and satisfies the cocycle condition over $\GS^{(2)}[p^{-1}]$, we call such an object a \textbf{rational Kisin descent datum}. Let $\mathrm{DD}_{\GS, [0,r]}^{\mathrm{ff}}$ denote the full subcateogry consisting of objects with $E$-height $\leq r$ and which are finite free over $\GS$. \newline
\end{definition}

\begin{remark}
The category $\mathrm{DD}_{\GS}$ is equivalent to the category of completed prismatic $F$-crystals on $\mathcal{R}$ introduced in \cite{Fcrystals}. 
\end{remark}

The main theorem of \cite{Fcrystals} is the following:

\begin{theorem}[\cite{Fcrystals} Prop. 3.25, Theorem 3.28]
\label{equivalencedescentrep}
    There is an equivalence of categories between $\mathrm{DD}_{\GS, [0,r]}$ and $\mathrm{Rep}_{\mathbb{Z}_p, [0,r]}^{\cris}$.
\end{theorem}

A main result of \cite{Fcrystals} is the association of a Kisin module to a crystalline representation:

\begin{theorem}{\cite[Theorem 4.19]{Fcrystals}}
\label{quasiKisinTheorem}
Let $V$ be a crystalline $\mathbb{Q}_p$ representation of $G_{\mathcal{R}}$ with Hodge-Tate weights in $[0,r]$, and let $\Lambda$ be a finite free $\mathbb{Z}_p$ lattice of $V$ stable under the $G_{\mathcal{R}}$ action. Let $\mathcal{M}$ be the  \'etale ($\varphi_{\etalering}$, $\etalering$)-module associated to $\Lambda$ as in Proposition \ref{latticeetalephi}. Then there exists a $\GS$ submodule $\GM$ of $\mathcal{M}$ stable under Frobenius such that:
\begin{itemize}
\item $\GM$ with the induced Frobenius is a Kisin module over $\GS$ of $E$-height $\leq r$.
\item $\etalering \otimes_{\GS} \GM = \mathcal{M}$ 
\item $\GM$ is equipped with the data of $\nabla_{\GM}$, a map $$\nabla_{\GM}: (\mathcal{R}_0 \otimes_{\varphi, \mathcal{R}_0} \GM/u\GM)[p^{-1}] \to (\mathcal{R}_0\otimes_{\varphi, \mathcal{R}_0} \GM/u\GM)[p^{-1}] \otimes_{\mathcal{R}_0} \widehat{\Omega}_{\mathcal{R}_0}$$ which is a topologically quasi-nilpotent connection commuting with Frobenius and satisfying rational $S$-Griffiths transversality. (See \cite[Def 4.1]{Fcrystals} and Definition \ref{SGriffiths} below.)
\end{itemize} 
\end{theorem}

If $\GM, \mathcal{M},$ and $\Lambda$ are as above, we call $\GM$ a \textit{Kisin module associated to} $\Lambda$. On ocassion we will want to forget the $\nabla$ structure on the Kisin module.

\begin{remark}
\label{freeremark}
In the setting of our base ring $R$, \cite[Remark 4.23]{Fcrystals} shows that the Kisin module associated to a lattice in a crystalline representation is projective. As $R$ is a local ring, the Kisin module is thus a free $\GS$-module. Furthermore, in the setting of a small base ring $T$, the Kisin module associated to a lattice in a crystalline representation is projective (but not necessarily free).
\end{remark}

\end{subsection}

\begin{subsection}{Base Change of the Kisin Module}
\label{BaseChangeSection}
Here we will again work over a general base ring $\mathcal{R}$ to state the base change theorems in full generality. \newline 

We recall the following standard fact from commutative algebra:

\begin{lemma}\cite[Lemma 3.1]{Fcrystals}
\label{flatlemma}
Let $A$ be a ring, $Q$ a flat $A$-module, and $N_1, N_2$ submodules of an $A$-module $N$. Then as submodules of $Q \otimes_A N$, we have $$Q \otimes_A (N_1 \cap N_2) = (Q \otimes_A N_1 ) \cap (Q \otimes_A N_2).$$
\end{lemma} 

\begin{subsubsection}{Shilov point base change of the Kisin module}

We will now recall the base change properties of the Kisin module. Let $\mathcal{R}'$ be another another general base ring for which there exists a map $\mathcal{R} \rightarrow \mathcal{R}'$ compatible with Frobenius which makes $\mathcal{R'}$ into an $\mathcal{R}$-module.

\begin{theorem}
\label{Kisinbasechange}
Let $(\GM, \varphi_{\GM})$ be the module and $\varphi_{\GM}$ data of the Kisin module associated to the lattice $\Lambda$ inside the crystalline representation $V$ as in Theorem \ref{quasiKisinTheorem} with $\nabla$ structure forgotten. Let $(\GM', \varphi_{\GM'})$ be the module and $\varphi_{\GM'}$ data of the Kisin module associated to $\Lambda\vert_{G_{\mathcal{R}'}}$ with $\nabla$ structure forgotten, and assume $\GM$ and $\GM'$ are finite free. Then $\GM \otimes_{\mathcal{R}} \mathcal{R}'$ gives the module and $\varphi_{\GM'}$ data of the Kisin module associated to $\Lambda\vert_{G_{\mathcal{R}'}}$, i.e. $\GM \otimes_{\mathcal{R}} \mathcal{R'} \cong \GM'$. 

\begin{proof}
Let $\mathcal{M}$ denote the \'etale $\varphi$-module associated to $\Lambda$ as in Theorem \ref{quasiKisinTheorem}.  Let $\GM'$ denote the Kisin module associated to $\Lambda\vert_{G_{\mathcal{R}'}}$ and $\mathcal{M}'$ its associated \'etale $\varphi$-module. We have injections $\GM \hookrightarrow \mathcal{M}$ and $\GM' \hookrightarrow \mathcal{M}'.$ By Proposition \ref{latticeetalephi}, we know that $\mathcal{M}' \cong \mathcal{O}_{\mathcal{E}, \mathcal{R}'} \otimes_{\mathcal{O}_{\mathcal{E}, \mathcal{R}}} \mathcal{M}$ and so we also have an injection $$\GM \otimes_{\GS_\mathcal{R}} \GS_\mathcal{R'} \hookrightarrow \mathcal{M}'.$$

Then $\GM \otimes_{\GS_\mathcal{R}} \GS_\mathcal{R'}$ and $\GM'$ are both the module data of Kisin modules lying inside of $\mathcal{M'}$. We claim there is unique such module data. By  \cite[Lemma 4.18]{Fcrystals} and Proposition \ref{latticeetalephi} there is a unique Kisin module $\GM_L$ associated to the \'etale $\varphi_{\mathcal{O}_{\mathcal{E}, L}}$-module $\mathcal{M}_L =\mathcal{M}' \otimes_{\mathcal{O}_{\mathcal{E}, \mathcal{R}'}} \mathcal{O}_{\mathcal{E}, \mathcal{R}_L'}$ and thus $$\GM \otimes_{\GS_\mathcal{R}} \GS_\mathcal{R'} \otimes_{\GS_{\mathcal{R}_L'}} \GS_{\mathcal{R}'_L}= \GM'\otimes_{\GS_{\mathcal{R}'}} \GS_{\mathcal{R}_L'}$$ as submodules of $\mathcal{M}_L.$ We also know that $$\GM \otimes_{\GS_\mathcal{R}} \GS_\mathcal{R'} \otimes_{\GS_{\mathcal{R'}}} \mathcal{O}_{\mathcal{E}, \mathcal{R}'} \cong \mathcal{M}' \cong \GM' \otimes_{\GS_{\mathcal{R}'}} \mathcal{O}_{\mathcal{E}, \mathcal{R}'}.$$ Thus we have $\GM \otimes_{\GS_{\mathcal{R}}} \GS_{\mathcal{R}'}$ contained in the intersection of $\GM' \otimes_{\GS_{\mathcal{R}'}} \GS_L'$ and $\GM' \otimes_{\GS_{\mathcal{R}'}} \mathcal{O}_{\mathcal{E}, \mathcal{R}'}$. By Lemma \ref{flatlemma} and the fact that $\GS_{\mathcal{R}_L'} \cap \mathcal{O}_{\mathcal{E}, \mathcal{R}'} = \GS_{\mathcal{R}'}$, this intersection is $\GM'$, concluding the proof as $\GM'$ and $\GM$ have the same rank and our maps are all compatible with the $\varphi$ structures. \newline
\end{proof}
\end{theorem}
\end{subsubsection}
Working in the setting of $R$, set $\GS_0 \coloneqq W(k)\llbracket u\rrbracket$ and let $S_0$ be the $p$-adically completed divided power envelope  of $\GS_0$ with respect to $E(u)$. We have induced maps $\GS_R \to \GS_0$ and $\mathcal{O}_{\mathcal{E}, R} \to \mathcal{O}_{\mathcal{E}, W(k)}$ which confirm that if $\mathcal{M}$ is an \'etale $(\mathcal\varphi_{\mathcal{O}_{\mathcal{E}, R}}, \mathcal{O}_{\mathcal{E}, R})$-module corresponding to a $\mathbb{Z}_p$-stable lattice inside a crystalline representation of $G_R$ then $\mathcal{M}_0 \coloneqq \mathcal{M} \otimes \mathcal{O}_{\mathcal{E}, W(k)}$ is an \'etale $(\mathcal\varphi_{\mathcal{O}_{\mathcal{E}, W(k)}}, \mathcal{O}_{\mathcal{E}, W(k)})$-module corresponding to a $\mathbb{Z}_p$-stable lattice inside a crystalline representation of $G_{K}.$ Then by \cite[Lemma 2.1.15]{Kisin} or as a corollary of Theorem \ref{Kisinbasechange}, we have \newline
\begin{corollary}
\label{quasiKisinclosed}
If $\GM$ is a Kisin module associated to a $\mathbb{Z}_p$-stable lattice in a crystalline representation of $G_R$, then $\GM \otimes_{\GS} \GS_0$ is, in the classical setting of \cite{Kisin}, a Kisin module associated to a $\mathbb{Z}_p$ stable lattice in a crystalline representation of $G_{K}$.
\end{corollary}

\begin{remark}
    Here we are using that $\GM$ is free by Remark \ref{freeremark} and that $r \in [0, p-2]$, which is why the corollary has only been stated for the base ring $R$.
\end{remark}

\end{subsection}
\begin{subsection}{Breuil Modules}
\label{BreuilModuleSection}
In this section we will work over $\T$ which is either a small base ring or the power series ring $R$. We define $S_\T$ to be the $p$-adically completed divided power envelope of $\GS_\T \coloneqq \T\llbracket u\rrbracket$ with respect to $E(u)$ a choice of minimal polynomial of $p$. The ring $S_\T$ comes equipped with a Frobenius map $\varphi_{S_\T}$ that extends $\varphi_{\GS_\T}.$ It also comes with its divided power filtration that we will denote $\Fil^iS_\T.$ Note that for $1 \leq i \leq p-1$, $\varphi_{S_\T}(\Fil^i S_\T) \subset p^iS$ (just consider the action of $\varphi_{S_\T}$ on $\frac{E(u)^i}{i!}$), so we can set $\varphi_{S_\T,i} = \frac{\varphi_{S_\T}}{p^i}: \Fil^i S_\T \to S_\T.$ The ring $S_\T$ also comes equipped with an integrable connection $\nabla_{S_\T}: S_\T \to S_\T \otimes_{\T} \widehat{\Omega}_\T.$\newline 

\begin{definition}
\label{Breuildefinition}
For $r \leq p-1$ a positive integer, let $\mathrm{Mod}_{{S}, \nabla}^{\mathrm{ff}, \mathrm{r}}(\T)$ be the category whose objects are quadruples $(\SM, \Fil^i(\SM), \nabla_{\SM}, \varphi_{\SM, i})$ where:
\begin{itemize} 
\item $\SM$ is a finite, projective $S_\T$-module.
\item  $\Fil^i(\SM)$ is a decreasing filtration of $\SM$ with $\Fil^0(\SM) = \SM$, $\Fil^{r+1}(\SM) \subset (\Fil^1 S_\T )\SM$, and $(\Fil^iS_\T)\Fil^i\SM \subset \Fil^{i+1}\SM.$
\item $\varphi_{\SM, i}$ are $\varphi_S$-semilinear maps $\varphi_{\SM, i}: \Fil^i(\SM) \to \SM$ so that the composite $\Fil^i(\SM) \to \Fil^{i-1}(\SM) \xrightarrow{\varphi_{\SM, i-1}} \SM$ is $p\varphi_{\SM, i}.$ Also, $\varphi_r(\Fil^r\SM)$ generates $\SM.$
\item $\nabla_{\SM}$ is a topologically quasi-nilpotent integrable connection which satisfies $S$-Griffiths transversality (see Definition \ref{SGriffiths} below) and commutes with each $\varphi_{\SM, i}$ as before. 

\end{itemize}
We refer to objects of this category as \textbf{Breuil modules}.
\end{definition}

Given a Kisin module $\GM$, we define the associated Breuil module to be $$\SM(\GM) \coloneqq \GM \otimes_{\GS_\T, \varphi_{\GS_\T}} S_\T.$$ If such a Kisin module exists for a given Breuil module $\SM$, we will say that ``$\SM$ arises from a Kisin module." In the next few sections, we show how the full data of a Breuil module can be constructed on $\SM \otimes_{\GS_\T, \varphi_{\GS_\T}} S_\T.$ 

\begin{subsubsection}{The filtration from the Kisin module}
\begin{definition}
\label{BreuilFilDef} Let $\SM$ be a Breuil module that arises from a Kisin module $\GM$ which arises from a crystalline representation. We define a decreasing filtration on $\mathscr{M}[p^{-1}]$ via:$$\Fil^i(\SM[p^{-1}]) \coloneqq \{m \in \SM[p^{-1}]~ | ~(1 \otimes \varphi_{\GM})(m) \in \Fil^i(S_\T[p^{-1}]) \otimes_{\GS_\T} \GM\}$$ and a filtration on $\SM$ via $$\Fil^i(\SM) \coloneqq \SM \cap \Fil^i(\SM[p^{-1}]).$$ \end{definition} We have the following lemma which shows that the filtration structure descends integrally:

\begin{lemma}
\label{FilinS}
Let $\SM$ be a Breuil module that arises from a Kisin module $\GM$ and let $0 \leq i < p-1$. We can express the filtration on $\SM$ as in Definition \ref{BreuilFilDef} as $$\Fil^i(\SM) = \{m \in \SM | (1 \otimes \varphi_{\GM})(m) \in \Fil^i(S_\T) \otimes_{\GS} \GM\}.$$ 
\end{lemma}

\begin{proof}
Observe that $\Fil^i(\SM) \subset \Fil^i(\SM[p^{-1}]),$ so we can realize an element of $\Fil^i(\SM)$ as an element $m \in \SM$ such that $(1\otimes \varphi_{\GM})(m)$ is contained in $$(\Fil^i(S[p^{-1}]) \otimes_{\GS_\T} \GM) \cap (S_\T \otimes_{\GS_\T} \GM) = (\Fil^i(S_\T[p^{-1}]) \cap S) \otimes_{\GS} \GM$$ where we are relying on Lemma \ref{flatlemma} (note $\GM$ is projective and hence flat). By comparing elements term-by-term we can show $\Fil^i(S_\T[p^{-1}]) \cap S_\T = \Fil^iS_\T$, as needed. Specifically, an element in $S_\T[p^{-1}]$ can be written as $$x = \sum_{k=0}^\infty \frac{E(u)^k}{k!} \sum_{j=0}^{\mathrm{deg}(E)-1} a_{kj} u^j$$ with $a_{kj} \in \T[p^{-1}].$ For this element to be in $\Fil^{i}(S_\T[p^{-1}])$ we must have $a_{kj} = 0$ whenever $0 \leq k < i $. But $x \in S_\T$ with each $a_{kj} = 0$ whenever $ 0 \leq k < i$, so $x \in \Fil^i(S_\T)$, as well.
\end{proof}
\end{subsubsection}

\begin{subsubsection}{The $\varphi_i$ structure} 

\begin{definition} Let $\SM$ be a Breuil module that arises from a Kisin module $\GM$. We can also define a $\varphi_{\SM, i}$ structure by setting $\varphi_{\SM, i}$ to be the composition: $$\varphi_{\SM, i} : \Fil^i(\SM) \xrightarrow{1 \otimes \varphi_{\GM}} \Fil^i(S_\T) \otimes_{\GS} \GM \xrightarrow{\varphi_{S_\T, i} \otimes 1} \SM.$$ 
\end{definition} 
Note the first arrow is well-defined by the definition of $\Fil^i(\SM).$ It is then clear that the composite $\Fil^i(\SM) \to \Fil^{i-1}(\SM) \xrightarrow{\varphi_{\SM, i-1}} \SM$ is $p\varphi_{\SM, i}$ because $\varphi_{S_\T,i} = \frac{\varphi_{S_\T}}{p^i}.$ \newline 
\end{subsubsection}
\begin{subsubsection}{The integrable connection} 
Here we discuss where the connection $$\nabla_{\SM[p^{-1}]}: \SM[p^{-1}] \to \SM[p^{-1}] \otimes_R \widehat{\Omega}_R$$ arises from, and then we will show it descends to the integral setting. \newline 

Let $\SM$ be a Breuil module that arises from a Kisin module with descent datum $\GM.$ By Theorem \ref{equivalencedescentrep} there exists $\Lambda$, a lattice inside a crystalline representation $V \coloneqq \Lambda \otimes_{\mathbb{Z}_p} \mathbb{Q}_p$ corresponding to $\GM$ under the equivalence of categories. Then we consider $D_{\cris}^\vee(V)$ which is a finite projective $\T[p^{-1}]$-module equipped with an integrable connection $\nabla_{D}$, a $\varphi$, and filtration structure which we will denote $\Fil^i(\Dcrisdual(V))$. Set $M = \T \otimes_{\varphi, \T} \GM/u\GM$. Then by \cite[Theorem 4.28]{Fcrystals} there is a $\varphi$-compatible isomorphism $M[p^{-1}] \to \Dcrisdual(V)$. By \cite[Theorem 4.2]{Fcrystals} we can identify $\mathscr{D} \coloneqq S_\T[p^{-1}] \otimes_{\T} M$ with $\SM[p^{-1}]$ as modules. Also see Lemma \ref{sectionstable}.\newline 

The connection $\nabla_D$ can be further extended to an integrable connection on $\mathscr{D} \coloneqq \Dcrisdual \otimes_{\T[p^{-1}]} S_\T[p^{-1}]$ defined as $\nabla_{\SM[p^{-1}]} \coloneqq \nabla_D \otimes 1 + 1 \otimes \nabla_{S_\T[p^{-1}]}$, and $\mathscr{D}$ comes equipped with a filtration defined inductively by setting $\Fil^0\mathscr{D} = \mathscr{D}$ and: $$\Fil^i\mathscr{D} \coloneqq \{x \in \mathscr{D} | N_u(x) \in \Fil^{i-1}\mathscr{D}, q(x) \in \Fil^i(\Dcrisdual(V))\},$$ where $N_u: \mathscr{D} \to \mathscr{D}$ is the $\T$-linear derivation $N_{u, S_\T} \otimes 1$ and $q(x)$ is the projection $\SM \to M$ as in \cite[Theorem 4.2]{Fcrystals}. By \cite[Lemma 4.31]{Fcrystals}, the filtration $\Fil^i(\SM[p^{-1}])$ defined in Definition \ref{BreuilFilDef} using the $\varphi$-structure on $\GM$ and the filtration $\Fil^i(\mathscr{D})$ defined above agree and, consequently, $\nabla_{\SM[p^{-1}]}$ satisfies rational $S_\T$-Griffiths transversality as we define below:

\begin{definition}
\label{SGriffiths}
 We say the connection $\nabla_{\SM[p^{-1}]}$ satisfies \textbf{rational} $S_\T$\textbf{-Griffiths transversality} if for every $i$ $$\partial_u(\Fil^{i+1}(\SM[p^{-1}]) \subset \Fil^i(\SM[p^{-1}])$$ and $$\nabla_{\SM[p^{-1}]}(\Fil^{i+1}(\SM[p^{-1}]) \subset \Fil^i(\SM[p^{-1}]) \otimes_R \widehat{\Omega}_R$$ where $\partial_u: \SM[p^{-1}] \to \SM[p^{-1}]$ is the derivation given by $\partial_{u, S_\T[p^{-1}]} \otimes 1.$ \newline

If $\nabla_{\SM[p^{-1}]}$ restricts to a connection $\nabla_{\SM} : \SM \to \SM \otimes_\T \widehat{\Omega}_{\T}$ where the same condition on the filtration holds with $\Fil^{\bullet} (\SM[p^{-1}])$ replaced by $\Fil^{\bullet}( \SM)$, we say that $\nabla_{\SM}$ satisfies \textbf{integral} $S_\T$\textbf{-Griffiths transversality}.
\end{definition}

We now show this connection descends to the integral setting. 

\begin{proposition} Let $\SM$ be a Breuil module that arises from a 
Kisin module $\ \GM$. Then $\SM$ is stable under $\nabla_{\SM}$ and satisfies integral $S$-Griffiths transversality. 
\label{BreuilNablaStable}
\begin{proof}
We consider the base change $\SM_L \coloneqq \SM \otimes_{S} S_{\T,L}$ to the Shilov point with imperfect residue field (defined in Section \ref{ShilovPointBaseChange}. Let $V$ be the crystalline representation corresponding to the \'etale $(\varphi_{\mathcal{O}_{\mathcal{E}}}, \mathcal{O}_{\mathcal{E}})$-module $\GM \otimes_{\GS_\T} \etalering$. By \cite[Theorem 4.20]{https://doi.org/10.48550/arxiv.2110.06001} there is a Breuil module $\SM_L'$ associated to $V\vert_{G_{L}}$. \newline 

By \cite[p. 10]{https://doi.org/10.48550/arxiv.2110.06001} there exists a Kisin module $\GM_L'$ such that $\SM_L' = \GM_L' \otimes_{\varphi_{\GS_{\T,L}}, \GS_{\T,L}} S_{\T,L}.$ By its uniqueness and Theorem \ref{Kisinbasechange} we must have $\GM_L' \cong \GM \otimes_{\GS_\T} \GS_{\T,L}$ and thus $\SM_L' \cong \GM \otimes_{\varphi, \GS_\T} S_{\T,L} \cong \SM_L.$ \newline

Thus $\SM_L$ is stable under its integrable connection $\nabla_{\SM_L} = \nabla_{\SM} \otimes 1 + 1 \otimes \nabla_{S_L}.$ Thus the image of $\SM$ under $\nabla_{\SM[p^{-1}]}$ is contained in $\SM[p^{-1}] \cap \SM_L.$ By Lemma \ref{flatlemma}:
\begin{align*}
\SM[p^{-1}] \cap \SM_L &= (\SM \otimes_S S_\T[p^{-1}]) \cap (\SM \otimes_{S_\T} S_{\T,L}) \\
&= \SM \otimes_{S_\T} (S_\T[p^{-1}] \cap S_{\T,L}) \\
&= \SM
\end{align*}
where $S_\T[p^{-1}] \cap S_{\T,L} = S_\T$ can be checked term-by-term as in the proof of Lemma \ref{sectionstable} below. \newline 

For integral $S$-Griffiths Transversality, we know $\nabla_{\SM}(\Fil^{i+1}(\SM)) \subset \Fil^i \SM \otimes_R \widehat{\Omega}_R$ as it satisfies rational $S$-Griffiths transversality and that $\nabla_{\SM}(\Fil^{i+1}(\SM)) \subset \SM \otimes_R \widehat{\Omega}_R$ because $\SM$ is stable under $\nabla_{\SM}$. We conclude that $\nabla_{\SM}$ satisfies integral $S$-Griffiths Transversality from Definition \ref{BreuilFilDef}.
\end{proof}
\end{proposition}
\end{subsubsection}

\begin{subsubsection}{The functor from Fontaine-Laffaille modules to Breuil modules}
In the classical setting, it has already been established as the main result of \cite[Theorem 1.1]{Gao} that there is a ``direct" equivalence of categories between Fontaine-Laffaille modules and Breuil modules. We rely on this equivalence when reducing the relative setting to the classical setting. It is also well-known that the same story holds in the classical case as in the above section. There is an equivalence of categories between Kisin modules and lattices inside of crystalline representations. If $\Lambda$ is a lattice associated to Kisin module $\GM$, there is a functor $D_{\cris}$ such that $\Dcris(\Lambda \otimes_{\mathbb{Z}_p} \mathbb{Q}_p) \cong (\GM/u\GM)[p^{-1}]$ where $M \coloneqq \GM/u\GM$ is a Fontaine-Laffaille module in the classical sense of \cite{Fontaine}. \newline

In the relative setting there is an analogous ``obvious" functor from Fontaine-Laffaille modules to Breuil modules. \newline

\begin{definition} \label{classicalfunctor} Let $M \in \mathrm{MF}_{\nabla, [0,r]}^{\mathrm{ff}}$. Then we can define the functor $\SM \coloneqq \SM(M) = M \widehat{\otimes}_\T S_\T$, the $p$-adic completion of $M \otimes_\T S.$ The additional structure is defined in the natural way as follows:
\begin{itemize}
\item $\varphi_\SM \coloneqq \varphi_M \otimes \varphi_{S_\T}$ 
\item $\nabla_\SM \coloneqq \nabla_M \otimes 1 + 1 \otimes \nabla_{S_\T}$
\item $\Fil^r(\SM) \coloneqq \sum_{i=0}^r \Fil^{i}(M) \widehat{\otimes}_{R} \Fil^{r-i}(S_\T)$
\item $\varphi_r \coloneqq \sum_{i=0}^r \varphi_{i, M} \otimes \varphi_{r-i, S_\T}$
\end{itemize}
where we've added an ``$M$" tag to a structure if it arises from that appropriate structure on $M$ and an ``$S$" tag if it arises from the appropriate structure on $S_\T$. Much of our focus will be on trying to reverse this functor: obtaining Fontaine-Laffaille module data from Breuil module data. \newline
\end{definition}

 In the classical setting, this functor is well-known to be an equivalence of categories and Hui Gao established a direct quasi-inverse in \cite[Theorem 1.1]{Gao}\newline

 \begin{remark}
     Our results will imply that this functor $M \mapsto \SM(M)$ is an equivalence of categories in the setting with base ring $\T$, as well. 
 \end{remark}

\end{subsubsection}
\end{subsection}
\begin{subsection}{Functors to Galois Representations}
\begin{subsubsection}{Definition of $\Tcris$}
\label{FLGaloisAction}
Let $M \in \mathrm{MF}^{\mathrm{ff}, [0,r]}_{\nabla}(R).$ Now we recall how the Galois action on $\Acris(R) \otimes_R M$ and hence $$\Tcris(M) \coloneqq (\Fil^r(A_{\mathrm{cris}} \otimes_{R} M))^{\varphi_r=1}$$ is defined. This is discussed in \cite[Section 2.3]{LMP}. Set $\varphi$ on $R$ to extend the natural $\varphi$ on $W(k)$ with $\varphi(t_i) = (1+t_i)^p -1.$ Choose $\widetilde{1+t_i} \in \overline{R}^\flat$ to be a fixed sequence of compatible $p^i$th roots of $1+t_i$ and denote by $[\widetilde{1+t_i}]$ its Teichm\"uller lift. For any $g \in G_R$ we define $\beta_i(g) = g([\widetilde{1+t_i}]) -[\widetilde{1+t_i}] \in \Fil^i\Acris(R)$, and we have for $a \otimes x \in \Acris(R) \otimes_R M$: $$g(a \otimes x) = \sum_I g(a)\gamma_I(\beta(g)) \otimes \nabla (\partial)^I(x)$$ where $\gamma_I(\beta(g))$ is defined as $\frac{\prod_{\ell} \beta_\ell(g)^{i_\ell}}{\prod_\ell i_\ell!}$ and $I = (i_1, \cdots , i_d)$ is a multi-index with $i_k \geq 0$ for all $k$ and $\nabla(\partial)^I = \left[\nabla(\frac{\partial}{\partial t_i})\right]^{i_1} \cdots \left[\nabla(\frac{\partial}{\partial t_d})\right]^{i_d}$. It follows from \cite[Lemma 2.20]{LMP} that this formula gives a well-defined $\AcrisR$-semilinear action of $G_R$ which is compatible with the filtration structure (i.e. $g(\AcrisR \otimes_R \Fil^rM) \subset \Fil^r(\AcrisR \otimes_R M)$) and commutes with the Frobenius structure.  This ensures $\Tcris(M)$ is well-defined and that $\Tcris(M) \in \mathrm{Rep}_{\mathrm{cris}, \mathbb{Z}_p}^{[0,r]}(G_{R})$. \newline

$\Tcris(M)$ can be defined analogously for $M \in \mathrm{MF}^{\mathrm{ff}, [0,r]}_{\nabla}(T),$ but $\varphi(t_i)$ is set to be $t_i^p$ and we use $\widetilde{t_i} \in \overline{T}^\flat$ instead of $\widetilde{1+t_i}$.
\end{subsubsection}
\begin{subsubsection}{The definition of $T_\GS(\GM)$}
\label{KisinGaloisAction}
Write $\GS$ for $\GS_R$. Let $(\GM, \varphi_{\GM}, f)$ be a Kisin module with descent data associated to a crystalline representation. Here we will work over our base ring $R$ and explain how the map $f$ as in Definition \ref{dddef} is used to construct a Galois action on $\GM \otimes_{\GS} \AinfR$ as in \cite[4.3]{DuLiu}. First recall that we have an embedding $R \hookrightarrow \AinfR$ via $t_i \mapsto [\widetilde{t_i+1}]-1$. \newline

Let $f_1: (\GS, (E)) \to (\AinfR, (E))$ be the map determined by $u \mapsto [\pi^\flat]$ which extends the above embedding of $R$ into $\AinfR$. Let $g \in G_R$ and let $f_2: (\GS, (E)) \to (\AinfR, (E))$ be the map determined by $u \mapsto g([\pi^\flat])$ and $t_i \mapsto g([\widetilde{t_i+1}]-1).$ Then the universal property described in Section \ref{KisinModulesSection} gives a unique map $f_g: (\GST, (E)) \to (\AinfR, (E))$ and since $f$ satisfies the cocycle condition there exists a map $d$ making the following diagram commute: \newline

\adjustbox{scale=1,center}{%
  \begin{tikzcd}[scale=1]   
  \GST \otimes_{p_1, \GS} \GM \arrow{r}{f} \arrow{d}{f_1} &\GST \otimes_{p_2, \GS} \GM \arrow{d}{f_g}\\
  \AinfR \otimes_{\GS} \GM \arrow{r}{d_g} & \AinfR \otimes_{g, \GS} \GM
  \end{tikzcd}} \newline

\noindent and from this we have a $G_R$ action on $\AinfR \otimes_{\GS} \GM.$ Alsom the $G_R$ action commutes with the $\varphi$-structure. \newline

We will make use of this $G_R$ action on two separate functors. We define $$T_{\GS}(\GM) \coloneqq (\GM \otimes_{\GS} W(\overline{R}^\flat\left[1/u\right])~)^{\varphi=1}$$  and $$T^r_{\GS}(\GM) \coloneqq (\Fil^r\varphi^*\GM \otimes_{\GS} \AinfR)^{\varphi_r=1}$$ where we write $\varphi^*\GM = \GS \otimes_{\varphi, \GS} \GM$ and we set $\Fil^r \varphi^*\GM = \{x \in \varphi^*\GM \vert (1 \otimes \varphi)(x) \in E(u)^r\GM\}$ and $\varphi_r: \Fil^r\varphi^*\GM \to \varphi^*\GM$ defined by $$\varphi_r(x) \coloneqq \frac{\varphi(x)}{\varphi(E(u)^r)}.$$

These functors are related via $$T^r_{\GS}(\GM) \cong T_{\GS}(\GM)(r)$$ as in \cite[Lemma 6.11]{LiLiu}, where $(r)$ denotes the $r$th Tate twist. Their proof only establishes the isomorphism in the classical case, but the proof works mutatis mutandis in the relative case. 
\begin{remark}
\label{dualremark}
We are using the covariant versions of these functors rather than the contravariant versions used in \cite{Fcrystals} (defined below) which are dual to the covariant versions. With this in mind, if we take a lattice $\Lambda \in \mathrm{Rep}_{\mathrm{cris}, \mathbb{Z}_p}^{[0,r]}(G_R)$, we should first consider $\Lambda^\vee(r)\in \mathrm{Rep}_{\mathrm{cris}, \mathbb{Z}_p}^{[0,r]}(G_R)$, the dual with the $r$th Tate twist applied. Then using Theorem \ref{quasiKisinTheorem} as stated for the contravariant versions of the functors, there exists a Kisin module with descent data $(\GM, \varphi_{\GM}, f)$ such that $T_{\GS}^\vee(\GM) = \Lambda^\vee(r)$ Then taking duals and applying the $r$th Tate twist, we obtain $T_{\GS}^r(\GM) = \Lambda$, so we do have a Kisin module with descent data $(\GM, \varphi_{\GM}, f)$ with $T_{\GS}^r(\GM) = \Lambda$ for our covariant version, too.
\end{remark}

\begin{remark}
\label{KisinIsomRem}
    We can define $T_\GS$ and $T_\GS^r$ analogously when working over a small base ring $T$ using the analogous embedding of $T \to \Ainf(T)$ via $t_i \mapsto [\widetilde{t_i}]$. Note that the embedding $T \to \Tatem$ induces an embedding $\Ainf(T) \to \Ainf(\Tatem)$ which gives an isomorphism $\T_{\GS_T}(\GM) \to T_{\GS_{\Tatem}}(\GM \otimes_{T} \Tatem)$ when $\GM$ is a Kisin module over $\GS_T.$
\end{remark}
\end{subsubsection}
\begin{subsubsection}{Equivalence between $T_{\GS}$ and $T_{\etalering}$}
Again let $\GM$ be a Kisin module associated to a crystalline representation. By \cite[Lemma 4.27]{Fcrystals}, the natural $G_{\tilde{R}_{\infty}}$-equivariant map $$ T_{\GS}^\vee(\GM) \coloneqq \Hom_{\GS, \varphi}(\GM, \widehat{\GS}^{\mathrm{ur}}) \to \Hom_{\etalering, \varphi}(\mathcal{M}, \widehat{\etalering}^{\mathrm{ur}}) = T_{\etalering}^\vee(\mathcal{M})$$ is an isomorphism, though we need to make slight adjustments because we are using the covariant versions of these functors as defined above rather than the contravariant versions sometimes used in \cite{Fcrystals}. \newline

\begin{lemma}
\label{etaleKisinGaloisEq}
The natural map $$T_{\etalering}(\mathcal{M}) \to T_{\GS}(\GM)$$ is an isomorphism of $G_{\tilde{R}_{\infty}}$ representations. 
\begin{proof}
Recall $\mathcal{M} = \GM \otimes_{\GS} \etalering.$ 

The embedding $f_1: \GS \to \AinfR = W(\overline{R}^\flat)$ above extends to an embedding $\widehat{\etalering}^{\mathrm{ur}} \to W(\overline{R}^\flat[1/u])$ which gives us a natural embedding $$T_{\etalering}(\mathcal{M}) = (\mathcal{M} \otimes_{\etalering} \widehat{\etalering}^{\mathrm{ur}})^{\varphi=1} = (\GM \otimes_{\GS}  \widehat{\etalering}^{\mathrm{ur}})^{\varphi=1} \hookrightarrow (\GM \otimes_{\GS} W(\overline{R}^\flat[1/u]))^{\varphi=1} = T_{\GS}(\GM).$$ The last embedding is known to be an isomorphism of $G_{\tilde{R}_{\infty}}$ representations as a consequence of \cite[Lemma 2.15]{Fcrystals}, completing the proof. 
\end{proof}
\end{lemma}
\end{subsubsection}
\begin{subsubsection}{The definition of $T_S(\SM)$}
\label{BreuilGaloisActionTwo}
Write $S$ For $S_R.$ Let $\SM$ be a Breuil module. We can define a Galois action on $\Acris(R) \otimes_S \SM$ is defined using $\nabla$ in the same way as the Galois action on a Fontaine-Laffaille module was defined above. First embed $S$ into $\Acris(R)$ where $u$ maps to $[\pi^\flat]$, the Teichmuller lift of a compatible sequence of $p$-power roots of $\pi$. If we let $g \in G_R$ and $a \otimes x \in \AcrisR \otimes_S \SM$ we set $$g(a \otimes x) = \sum g(a) \partial_u^{j_0} \partial_{t_1}^{j_1} \cdots \partial_{t_d}^{j_d}(x) \cdot \gamma_{j_0}(g[\pi^\flat]-[\pi^\flat])\prod_{i=1}^d \gamma_{j_i}(g([t_i^\flat])-[t_i^\flat])~~~~~~~~~~~~~(*)$$ with notation as above in the previous section and the sum over multi-indices $(j_0, \cdots , j_d)$ of nonnegative integers. Similarly to Section \ref{FLGaloisAction}, this gives us an $\Acris(R)$-semilinear action of $G_R$ which preserves the filtration structure ($g(\Fil^r(\AcrisR \otimes_S \SM) \subset \Fil^r(\AcrisR \otimes_S \SM$)) and commutes with Frobenius. It is also clear that if $\SM = M \otimes_R S$, then $$T_S(\SM) \coloneqq \Fil^r(\AcrisR \otimes_S \SM)^{\varphi_r=1} = \Fil^r(\AcrisR \otimes_S (M \otimes_R S))^{\varphi_r=1} = \Tcris(M). $$

We can also view this Galois action on $\Acris(R) \otimes_S \SM$ determined by $\nabla$ as one arising from the descent data. We first recall that there are rings $S^{(1)}$ and $S^{(2)}$ analogous to $\GST$ and $\GS^{(2)}$. The details of the construction can be found in \cite[Example 3.9]{Fcrystals}. We will write $p_1$ and $p_2$ for the maps from $S$ to $S^{(1)}$ analogous to the $p_1$ and $p_2$ defined in Section \ref{KisinGaloisAction} \newline

Given $(\GM, \varphi_{\GM}, f)$ a
Kisin module associated to a lattice in a crystalline representation as in Theorem \ref{quasiKisinTheorem} , \cite[Construction 4.3]{Fcrystals} constructs a descent datum $f_S$ on $\SM[p^{-1}] = S[p^{-1}] \otimes_R M$ as follows: \newline

Let $\partial_u: S[p^{-1}] \otimes_R M \to S[p^{-1}] \otimes_R M$ be the derivation given by $\partial_{u, S} \otimes 1$ and let $\partial_{t_i}$ be the derivation given by $\partial_{t_i, S} \otimes 1 + 1 \otimes \partial_{ti, M}$. Then we can define $f_S: S^{(2)} \otimes_{p_1, S} \SM[p^{-1}] \to S^{(2)} \otimes_{p_2, S} \SM[p^{-1}]$ as $$f_S(x) = \sum_J \partial_u^{j_0}\partial_{t_1}^{j_1} \cdots \partial_{t_d}^{j_d}(x) \gamma_{j_0}(p_2(u)-p_1(u)) \prod_{i=1}^d \gamma_{j_i}(p_2(t_i) -p_1(t_i)), ~~~~~~~~~~~~~(**)$$ where the sum is over multi-indices $J = (j_0, \cdots , j_d)$ of nonnegative integers. \newline

This descent datum determines a Galois action on $\Acris(R) \otimes \SM$ in the same way as the descent datum $f$ determined a Galois action on $\AinfR \otimes_{\GS} \varphi^* \GM$ just with $\GST$ and $\AinfR$ replaced by $S^{(2)}$ and $\AcrisR$. Specifically, let $g \in G_R.$ The maps $f_1$ and $f_2$ from Section \ref{KisinGaloisAction} extend to embeddings $S \to \AcrisR$ and by the analogous universal property as discussed at the beginning of Section \ref{KisinModulesSection}, there exists a unique map $f_{S, g}: (S^{(1)}, (E)) \to (\Acris(R), (E))$ and since $f_S$ satisfues the cocycle condition, there exists a map $d$, depending on $g$, making the diagram commute: \newline

  \adjustbox{scale=1,center}{%
  \begin{tikzcd}[scale=1]   
  S^{(1)} \otimes_{p_1, S} \SM \arrow{r}{f_S} \arrow{d}{f_1} &S^{(1)} \otimes_{p_2, S} \SM \arrow{d}{f_{S,g}}\\
  \AcrisR \otimes_{S} \SM \arrow{r}{d_g} & \AcrisR \otimes_{g, S} \SM
  \end{tikzcd}}\newline
  
which establishes the desired Galois action on $\AcrisR \otimes_S \SM$. \newline 

So $f_S$ determines a Galois action of $G_R$ on $\AcrisR \otimes_S \SM$ and hence on $$T_S(\SM) \coloneqq \Fil^r(\AcrisR \otimes_S \SM)^{\varphi_r=1}.$$

We have seen two ways of defining a Galois action on $T_S(\SM)$: the first via the descent datum $f_S$ and the second via the connection $\nabla$. Via the explicit formulas $(*)$ and $(**)$ above, we see that the two ways of defining the Galois action are the same.

\begin{remark}
We define $T_S$ analogously when working over a small base ring $T$ instead of $R$. Writing $S$ for $S_T$ and letting $\SM$ be a Breuil module over $S$, we have $T_S(\SM) \coloneqq \Fil^r(\Acris(T) \otimes_S \SM)^{\varphi_r=1}.$
\end{remark}
\end{subsubsection}

\begin{subsubsection}{Equivalence between $T_{\GS}^r$ and $T_S$}
We will now explore the relationship between $T_{\GS}^r(\GM)$ and $T_S(\SM)$ defined in the preceding sections. Continue to assume that $(\GM, \varphi_{\GM}, f)$ is a Kisin module arising from a $\mathbb{Z}_p$-crystalline representation of $G_R$. Let $\SM$ be a Breuil module arising from $(\GM, \varphi_{\GM}, f)$ as in Section \ref{BreuilModuleSection} (recall this means $\SM = \SM(\GM) \coloneqq \GM \otimes_{\GS, \varphi_{\GS}} S$).\newline 

Now consider the natural map $$\alpha: \AinfR \otimes_{\GS} \varphi^*\GM  \to  \AcrisR \otimes_S \SM.$$ 
  
  By \cite[Prop. 4.6]{Fcrystals}, the map $f_S$ as defined in the previous section (originally constructed in \cite[Construction 4.3]{Fcrystals}) satisfies $$f_S = S^{(1)} \otimes_{\varphi, \GS^{(1)}} f.$$ Thus we see that $f_S$ and $S^{(1)} \otimes_{\varphi, \GS^{(1)}} f$ define the same Galois action on $\AcrisR \otimes_S \SM$, as we discussed how $f$ defines a Galois action on $\AinfR \otimes \varphi^*\GM$ using the following diagram:\newline

\adjustbox{scale=1,center}{%
  \begin{tikzcd}[scale=1]   
  \GST \otimes_{p_1, \GS} \GM \arrow{r}{f} \arrow{d}{f_1} &\GST \otimes_{p_2, \GS} \GM \arrow{d}{f_g}\\
  \AinfR \otimes_{\GS} \GM \arrow{r}{d_g} & \AinfR \otimes_{g, \GS} \GM
  \end{tikzcd}} \newline

  and $f_S$ defines a Galois action on $\AcrisR \otimes \SM$ via the same diagram mutatis mutandi: \newline

  \adjustbox{scale=1,center}{%
  \begin{tikzcd}[scale=1]   
  S^{(1)} \otimes_{p_1, S} \SM \arrow{r}{f_S} \arrow{d}{f_1} &S^{(1)} \otimes_{p_2, S} \SM \arrow{d}{f_{S,g}}\\
  \AcrisR \otimes_{S} \SM \arrow{r}{d_g} & \AcrisR \otimes_{g, S} \SM
  \end{tikzcd}}\newline

Thus we see that the Galois action established on $T_{\GS}^r(\GM)$ is compatible with the Galois action on $T_S(\SM)$ via the natural map $\alpha$.
  
  As $\SM = \SM(\GM)$, it is standard to see from the definitions of the filtration structure on $\GM$ and $\SM$ that the map $\alpha$ respects the filtration and $\varphi_r$ structures. This gives us a natural injective map $$T_{\GS}^r(\GM) \to T_S(\SM)$$ compatible with the $G_R$-action, filtration, and $\varphi_i$, structure. \newline

We will later prove the following theorem:

\begin{theorem}
\label{TGTScompat}
The natural injection $$T_{\GS}^r(\GM) \to T_S(\SM) \coloneqq \Fil^r(\AcrisR \otimes_S \SM)^{\varphi_r=1}$$ induced by the embedding $\AinfR \to \AcrisR$ is an isomorphism of $G_{R}$ representations.  
\end{theorem}
but we will need information about the Fontaine-Laffaille module associated to $\SM$ which won't be established until the next section. Nevertheless, we state the result here in our section on Galois representations for completeness.
\end{subsubsection}

\end{subsection}
\begin{subsection}{A summary of categories}
Here we include a brief diagram summarizing the relationship between the many categories introduced above in the setting of our base ring $R$:

 \adjustbox{scale=1,center}{%
  \begin{tikzcd}[scale=1]    
 ~ & \mathrm{Rep}_{\mathbb{Z}_p, [0,r]}^{\cris}(G_{R}) &\arrow{l}[above]{\Tcris} \mathrm{MF}_{\nabla,[0,r]}^{\mathrm{ff}} \arrow{d}{-~ \widehat{\otimes}_R S}\\
  ~ \mathrm{Mod}_{\etalering}^{\mathrm{ff}} \arrow{ur}{T_{\etalering}}& \arrow{l}[below]{\etalering \otimes_{\GS}~-} \mathrm{DD}_{\GS, [0,r]}^{\mathrm{ff}} \arrow{r}[below]{-~\otimes_{\GS, \varphi} S}  \arrow{u}{T_{\GS}^r}[right]{\cong}&  \mathrm{Mod}_{S, \nabla}^{\mathrm{ff},~r} \arrow{ul}[above]{T_S}\\
  \end{tikzcd}
} 
\end{subsection}
\begin{subsection}{$T_{\cris}$ is fully faithful}
To conclude Section 2, we will show that $\Tcris$ is fully faithful. First we work over the base ring $R$. \newline 

Let $M_1$ and $M_2$ be Fontaine-Laffaille modules over $R$ so that there exists a morphism $$\psi: \Tcris(M_1) \to \Tcris(M_2)$$ of $G_R$ representations. By Proposition \ref{latticeetalephi} there is an equivalence of categories between $\mathbb{Z}_p$ stable lattices in finite free $G_{R_{\infty}}$ representations and \'etale $(\varphi_{\etalering} , \etalering)$-modules and thus exists a unique morphism $\psi_{\mathcal{M}}: \mathcal{M}_1 \to \mathcal{M}_2$ where $\mathcal{M}_i$ is the \'etale $(\etalering, \varphi_{\etalering})$-module associated to the $G_\T$ representation $\Tcris(M_i).$ \newline 
 
Moreover, there is a unique morphism $\psi_{\GM}: \GM_1 \to \GM_2$ between the associated Kisin modules with descent data by \cite[Proposition 3.25, Theorem 3.28]{Fcrystals}. Proposition 3.25 establishes an equivalence of categories between $\mathrm{DD}^f_{\GS}$ and a category of prismatic $F$-crystals, and Theorem 3.28 establishes that this category of prismatic $F$-crystals is equivalent to $\mathrm{Rep}_{\mathbb{Z}_p, [0,r]}^{\mathrm{cris}}(G_\T)$. So it suffices to show that $\psi_{\GM}$ induces a map $\psi_M: M_1 \to M_2$ such that $T_{\mathrm{cris}}(\psi_M)$ agrees with $T_{\GS}^r(\psi_{\GM})$ on $\Tcris(M_1) \cong T_{\GS}^r(\GM_1)$, which are isomorphic as established above. \newline

Set $\psi_{\SM} = \psi_{\GM} \otimes_{\GS, \varphi_{\GS}} S: \SM_1 \to \SM_2$ which is a morphism of the associated Breuil modules. We can then obtain a morphism on the associated Fontaine-Laffaille modules using the unique sections guaranteed by Theorem \ref{sectionstable}. Let $s_1$ be the section $M_1 \to \SM_1$ and $s_2$ the section $M_2 \to \SM_2$ as in the following diagram: 

 \adjustbox{scale=1,center}{%
  \begin{tikzcd}[scale=1]    
      \SM_1 \arrow{r}{\psi_{\SM}} & \SM_2\\
      M_1 \arrow{r}{\psi_M} \arrow{u}{s_1}& M_2 \arrow{u}{s_2}
  \end{tikzcd}
} 

where the bottom map $\psi_M$ is defined to be the natural projection $\SM_2 \to M_2$ composed with $\psi_{\SM} \circ s_1$. Tensoring the above diagram with $\Acris,$ we see that $\Tcris(\psi_M)$ agrees with $T_S(\psi_{\SM}).$ If we let $g_i: \GM_i \to \SM_i$ denote the map $\GM_i \to \GM_i \otimes_{\GS, \varphi_{\GS}} S$, we obtain a similar diagram 

 \adjustbox{scale=1,center}{%
  \begin{tikzcd}[scale=1]    
      \GM_1 \arrow{r}{\psi_{\GM}} \arrow{d}{g_1} & \GM_2\arrow{d}{g_2}\\
      \SM_1 \arrow{r}{\psi_{\SM}} & \SM_2 
  \end{tikzcd}
} 

which shows that $T_{\GS}^r(\psi_{\GM})$ will agree with $T_S(\psi_{\SM}).$ \newline

To complete the proof of fully-faithfulness, we need to justify the uniqueness of $\psi_M$. This follows from the established rational theory. Such a morphism $\psi_M: M_1 \to M_2$ induces a morphism $\psi_{M[p^{-1}]}: M_1[p^{-1}] \to M_2[p^{-1}] $ which is known to be unique as \cite[Proposition 4.28]{Fcrystals} establishes an isomorphism $M_i[p^{-1}] \to D_{\cris}^\vee(T_\cris(M_i) \otimes_{\mathrm{Z}_p} \mathbb{Q}_p)$ and the induced morphism on $D_\cris^\vee$ is known to be unique by the original theory of Brinon in \cite{BrinonPeriodRings}. 

Thus we have established:
\begin{theorem}
The functor $$\Tcris: \mathrm{MF}^{\mathrm{ff}}_{\nabla, [0, r]}(R) \to \mathrm{Rep}_{\mathbb{Z}_p, [0, r]}^{\mathrm{cris}}(G_R)$$ given by $$\Tcris(M) \coloneqq (\Fil^r(A_{\mathrm{cris}}(R) \otimes_{R} M))^{\varphi_r=1}$$ is fully faithful. 
\end{theorem}

We can also establish fully-faithfulness for small base rings via base change. \newline

We also have the analogous result for the Tate Algebra: 
\begin{theorem}
The functor $$\Tcris: \mathrm{MF}^{\mathrm{ff}}_{\nabla, [0, r]}(T) \to \mathrm{Rep}_{\mathbb{Z}_p, [0, r]}^{\mathrm{cris}}(G_T)$$ given by $$\Tcris(M) \coloneqq (\Fil^r(A_{\mathrm{cris}}(T) \otimes_{T} M))^{\varphi_r=1}$$ is fully faithful. 

\begin{proof}
    Let $M_1$ and $M_2$ be Fontaine-Laffaille modules over $T$ so that there exists a morphism $$\psi: \Tcris(M_1) \to \Tcris(M_2)$$ of $G_T$ representations. Write $\hatM$ for $M \otimes_T \Tatem.$ We get an induced map $\Tcris(\hatM_1) \to \Tcris(\hatM_2)$, for which there exists a unique morhphism $\hatM_1 \to \hatM_2$  as $\Tcris$ is fully faithful with the base ring $R$. We also have a unique induced map $D_1 = M_1[p^{-1}] \to M_2[p^{-1}] = D_2$, setting up the following diagram:  
    \adjustbox{scale=1,center}{%
  \begin{tikzcd}[scale=1]    
      \hatM_1 \arrow{r}{\varphi}& \hatM_2 \\
      M_1 \arrow{u}& M_2 \arrow{u}
  \end{tikzcd}
} 
where the vertical maps are injective. But the image of $M$ under $\varphi$ is contained in $\hatM_2 \cap M_2[p^{-1}] = M_2$, completing the proof of fully faithfulness. 
\end{proof}
\end{theorem}



\end{subsection}
\end{section}

\begin{section}{Essential Surjectivity of $T_{\cris}$ over Power Series Rings}
In this section we aim to prove the following two theorems, working over the base ring $R$. Let $I_0 \subset S$ denote the kernel of the projection $S \to R$ given by $u \mapsto 0.$

\begin{theorem}
\label{MisFL}
Let $\SM$ be a Breuil module which arises from a Kisin module which arises from a $\Lambda,$ an object in $\mathrm{Rep}_{\mathbb{Z}_p}^{\cris [0,r]}(G_R).$  Set $$M \coloneqq \SM/I_0\SM \cong \GM/u\GM \otimes_{R, \varphi_{\GS}} R.$$ Then $M$ carries the data of a Fontaine-Laffaille module and satisfies $M \otimes_R S \cong \SM$ as in Definition \ref{classicalfunctor} We call it the ``Fontaine-Laffaille module associated to $\Lambda$."
\end{theorem}

\begin{theorem}
Let $M$ be the Fontaine-Laffaille module associated to $\Lambda.$ Then $$T_{\mathrm{cris}}(M) \cong \Lambda.$$ 
\end{theorem}

Combined, these theorems complete the proof of essential surjectivity of $T_{cris}.$ For the remainder of Section 3, we will assume $\SM$ satisfies the hypotheses of the above theorem.

\begin{subsection}{Fontaile-Laffaille data on $M$}
We will now define the Fontaine-Laffaille data on $M$ and show that $M$ is indeed a Fontaine-Laffaille module with our key tool being Theorem \ref{genericclosed}. As  modules, we have from Theorem \ref{Kisinbasechange} and Proposition \ref{quasiKisinclosed} that $M_g \coloneqq M \otimes_{R} R_g$ and $M_0 \coloneqq M \otimes_{R} W(k)$ are isomorphic, as modules, to Fontaile-Laffaille modules arising from $\mathbb{Z}_p$ stable lattices in crystalline representations, so they come equipped with the data of a classical Fontaine-Laffaille module as in Definition \ref{FLdefClassic}. We do not yet know, though, whether this data, particularly the filtration data, is compatible with the base change maps. Showing this will be a major focus of this section.
\begin{subsubsection}{The Filtration and $\varphi_i$ structure on $M$}

We first develop the filtration and $\varphi_i$ structure on $M \coloneqq \mathscr{M}/I_0\mathscr{M} \cong \GM/u\GM \otimes_{R, \varphi} R$. The key is to show that the following Lemma from \cite{Fcrystals} descends to the integral situation:

\begin{lemma}[\cite{Fcrystals} Lemma 4.2]
\label{section}
Consider the projection $q: \SM \to M$ induced by the $\varphi$-compatible projection $S \to R$, $u \mapsto 0.$ Then $q$ admits a unique $\varphi$-compatible section $s: M[p^{-1}] \to \SM[p^{-1}].$ Furthermore, $1 \otimes s: S[p^{-1}] \otimes_{R[p^{-1}]} M[p^{-1}] \to \SM[p^{-1}]$ is an isomorphism.
\end{lemma}

We first prove the following:

\begin{lemma}
\label{sectionstable}
The projection $q$ as in Lemma \ref{section} admits a unique $\varphi$-compatible section $$s: M \to \SM.$$

\begin{proof}
To attain this section we base change to the Shilov point obtaining the following commutative diagram: \newline

\adjustbox{scale=0.85,center}{%
  \begin{tikzcd}[scale=1]
    M \arrow[hook]{r} \arrow[hook]{d} & M[\frac{1}{p}]\ \arrow{r}{s} \arrow[hook]{d} & \SM[\frac{1}{p}] \arrow[hook]{d} \\
    M_g \arrow{r} & M_g[\frac{1}{p}] \arrow[hook]{r}{s \otimes R_g} & \SM_g[\frac{1}{p}] \\
  \end{tikzcd}
}

Note that $s \otimes R_g$ determines a section $M_g[\frac{1}{p}] \to \SM_g[\frac 1p]$. By \cite[Prop 3.2.3]{Gao} this section is unique and it takes $M_g$, and thus $M$, into $\SM_g$. Tracing $M$ through the top of the diagram, we see that the image of $M$ in $\SM_g[\frac 1p]$ is contained in $\SM[\frac 1p].$ 

From Lemma \ref{flatlemma}, we obtain:
\begin{align*}
\SM_g \cap \SM\left[\frac 1p\right] &= \SM \otimes \left(S_g \cap S\left[\frac 1p\right]\right) \\
&= \SM \otimes S.
\end{align*}

We can verify $S_g \cap S\left[\frac 1p\right] = S$ term by term. If we take $x$ in the intersection, we can write it as $$x = \sum \frac{E(u)^i}{i!}\sum a_{ij}u^j$$ with $a_{ij} \in W(k_g).$ We can also write it as $$x = \sum \frac{E(u)^i}{i!} \sum b_{ij} u^j$$ with $b_{ij} \in W(k)[\frac 1p].$ Comparing term-by-term we can see that $a_{ij} = b_{ij}.$ 
\end{proof}
\end{lemma} 

\begin{lemma}
\label{NakSec}
The map $1 \otimes s: S \otimes_{R} M \to \SM$ is an isomorphism.
\begin{proof}
Since $S \otimes_R M$ and $\SM$ are finite, free $S$-modules of the same rank, it suffices to show that $1 \otimes s$ is surjective. Moreover, $S$ is a local ring, so by Nakayama's Lemma it suffices to show that $1 \otimes s$ is surjective after reducing modulo the maximal ideal of $S$ which is the ideal generated by $p, t_1, \cdots, t_d$, $u$, and the divided powers $\frac{E(u)^i}{i!}$. After reducing modulo the maximal ideal, $1 \otimes s$ becomes the identity map which is obviously surjective. 
\end{proof}
\end{lemma}

We now define the filtration and $\varphi_i$ structure on $M$. 

\begin{definition}Realizing $M$ as a submodule of $\SM$ via the section $s$, we can define $$\Fil^i(M) \coloneqq \Fil^i\SM \cap M.$$ 
\end{definition}

\begin{definition}We set $\varphi_{M, i} = \varphi_{\SM, i}\big|_{M}$.\end{definition}

Note that the filtration on $\SM$ induces a filtration $\Fil^i(M)$ on $M$ for which the composite $\Fil^i(M) \to \Fil^{i-1}(M) \xrightarrow{\varphi_{M, i-1}} M$ is $p\varphi_{M, i}$. By the corresponding properties on $\Fil^i \SM$  (see Definition \ref{Breuildefinition}), it is easy to see that $\Fil^0M = M$ and $\Fil^{r+1} M = \{0\}.$  \newline

Thus $M$ has the structure of an object in $\mathrm{MF}_{\mathrm{big}}^{\mathrm{f}}(R)$. We rely on Proposition \ref{genericclosed} to show that we have an object of $\mathrm{MF}^{\mathrm{ff}}(R)$ in the following section.
\end{subsubsection}
\end{subsection}

\begin{subsection}{$M \otimes_R R_g \cong M_g$ as Fontaine-Laffaille modules}

We need to show that the filtration data on the tensor product $M 
\otimes_R R_g$ is compatible with the filtration on $M_g$. The difficulty in doing this is establishing a base change theorem for the Breuil module $\SM$. The ring $S_g$ is very large, so treating $\SM$ as an $R$-module and tensoring with $R_g$ over $R$ will not get us back $\SM_g$: we would need to $p$-adically complete $\SM \otimes_R R_g.$ The fact $S$ is non-Noetherian makes $p$-adic completion difficult to work with. \newline

Instead, we can work over the ring $S/\Fil^i(S)$ which is Noetherian when $i < p$. This will allow us to decompose $\Fil^i(\SM_g)$ and show that $\Fil^i(M_g)$ is contained in the nicer piece.

\begin{lemma}
\label{genericdecomp}
Let $i$ be a nonnegative integer less than $p$. Then we can decompose $\Fil^i(\SM_g)$ as follows: $$\Fil^i(\SM_g) = \Fil^i \SM \otimes_{R} R_g + \SM \otimes_{S} \Fil^i(S_g).$$
\begin{proof}
Another way to express the definition of $\Fil^i \SM$ is as a kernel in the exact sequence 
\[0 \to \Fil^i \SM \to \SM \to \left(S/\Fil^iS\right) \otimes_{\GS} \GM.\tag{1}\] 
It follows quickly from the definition of $\Fil^i\SM$ that $\SM \otimes_S \Fil^i(S)$ is contained in the kernel, so we get an exact sequence $$0 \to \Fil^i \SM/\left(\SM \otimes_S \Fil^i(S)\right) \to \SM/\left(\SM \otimes_S \Fil^i(S)\right) \to \left(S/\Fil^iS\right) \otimes_{\GS} \GM.$$
 We can now tensor with $R_g$ which maintains the exactness since $R_g$ is a flat $R$-module. As a further conesequence of the flatness of $R_g$, the tensor product commutes with quotients and we obtain the exact sequence: 

\[
0 \to \frac{\Fil^i \SM \otimes_R R_g}{\SM \otimes_S (\Fil^i(S) \otimes_R R_g)} \to \frac{\SM}{\SM \otimes_S (\Fil^i(S))} \otimes_R R_g \to \left( \left(S/\Fil^iS\right) \otimes_\GS \GM \right)\otimes_R R_g \tag{2}
\]

We also have on the Shilov point the exact sequence $$0 \to \Fil^i\SM_g \to \SM_g \to (S_g/\Fil^iS_g) \otimes_{\GS_g} \GM_g$$ where the rightmost map restricts to the rightmost map of exact sequence $(1)$ via the embedding $\SM \hookrightarrow \SM_g$. We can also see that $\SM \otimes_S \Fil^i S_g$ is in the kernel, so this gives an exact sequence \[0 \to \frac{\Fil^i\SM_g}{\left(\SM \otimes_S \Fil^i S_g\right)} \to \frac{\SM_g}{\left(\SM \otimes_S \Fil^i S_g\right)} \to (S_g/\Fil^iS_g) \otimes_{\GS_g} \GM_g.\tag{3}\] 

Now note that $$\frac{\SM_g}{\SM \otimes_S \Fil^iS_g} \cong \frac{\SM \otimes_S S_g}{\SM \otimes_S \Fil^iS_g} \cong \SM \otimes_S \frac{S_g}{\Fil^iS_g}$$ since $\SM$ is a finite, free $S$-module. Also, $$\frac{\SM}{\SM \otimes_S \Fil^iS} \otimes_R R_g\cong \SM \otimes \frac{S}{\Fil^i(S)} \otimes_R R_g.$$ Since $S/\Fil^iS$ is a finite, free $R$-module and $S_g/\Fil^iS_g$ is a finite, free-$R_g$ module of the same rank, the rightmost maps of $(2)$ and $(3)$ are the same with the same domain. Thus the kernels must be the same. \newline

Since $\Fil^iS \otimes_R R_g \subset \Fil^i(S_g)$, we obtain that an element in $\Fil^i \SM_g$ is determined by an element of $\Fil^i\SM \otimes_R R_g$ up to an element in $\SM \otimes_S \Fil^iS_g$, completing the proof of the lemma. 
\end{proof} 
\end{lemma} 

We now want to use this decomposition to show that elements of $\Fil^i(M_g)$ cannot arise containing any $\Fil^i(S_g) \otimes_S \SM$ components. Formally, we have the following lemma: 
\begin{lemma}
\label{genericcontainment}
$$\Fil^i(M_g) \subset \Fil^i\SM \otimes_R R_g$$
\begin{proof}
Let $x \in \Fil^i(M_g)$, which we know to be a subset of $\Fil^i(\SM_g).$ This means by Lemma \ref{genericdecomp} we can write $x = y + z$ with $y \in \Fil^i(S_g) \otimes_S \SM$ and $z \in \Fil^i\SM \otimes_R R_g$. Since  $\Fil^i(M_g) \subset M_g \cong M \otimes_R R_g$, we have that $y = z-x \in \SM \otimes_R R_g.$ We claim $y \in (\Fil^i(S) \otimes_S \SM) \otimes_R R_g \subset \Fil^i \SM \otimes_R R_g.$ \newline

Let $e_1, e_2, \cdots , e_n$ be an $S$-basis of $\SM.$ Since $y \in \Fil^i(S_g) \otimes_S \SM$, we can write $$y = \sum_{m=1}^n s_me_m$$ with $s_m \in \Fil^i(S_g).$ Since $y \in \SM \otimes_R R_g$, we can write 
\begin{align*}
y &= \sum_{j=1}^\ell \left[\sum_{m=1}^n \left(\sum_{k=1}^\infty a_{k, m, j} \frac{E(u)^k}{k!} \right)e_m \right] \otimes r_j \\
&= \sum_{m=1}^n \left[\sum_{j=1}^\ell \sum_{k=1}^\infty \left((a_{k, m, j} \otimes r_j) \frac{E(u)^k}{k!}\right) \right] e_m \\
&= \sum_{m=1}^n \left[\sum_{k=1}^\infty \left(\sum_{j=1}^\ell  (a_{k, m, j} \otimes r_j)  \right)\frac{E(u)^k}{k!} \right]e_m \\
\end{align*}

with $a_{k, m, j} \in R[u]$ and $r_j \in R_g.$ Thus,  $$s_m = \sum_{k=1}^\infty \left(\sum_{j=1}^\ell  (a_{k, m, j} \otimes r_j)  \right)\frac{E(u)^k}{k!}.$$ Then since $s_m \in \Fil^i(S_g)$, we must have $$\left(\sum_{j=1}^\ell  (a_{k, m, j} \otimes r_j)  \right) = 0$$ for $k < i,$ and then it is clear that $y \in (\Fil ^i(S) \otimes_S \SM) \otimes_R R_g \subset \Fil^i \SM \otimes_R R_g$ and the lemma is proven.
\end{proof}
\end{lemma}

\begin{theorem} 
\label{genericfil}
$\Fil^i(M_g) \cong \Fil^i(M) \otimes_R R_g.$ 
\begin{proof}
We know that $\Fil^i(M_g) \subset M \otimes_R R_g \cong M_g$ and, by Lemma \ref{genericcontainment}, $\Fil^i(M_g) \subset \Fil^i \SM \otimes_R R_g.$ By Lemma \ref{flatlemma}, we have:
\begin{align*}
\Fil^i(M_g) &\subset \left(\Fil^i \SM \otimes_R R_g\right) \cap \left( M \otimes_R R_g \right) \\
&= \left(\Fil^i \SM \cap M\right) \otimes_R R_g \\
& = \Fil^i(M) \otimes_R R_g.
\end{align*}

The inclusion in the other direction is easy to establish.  Clearly $\Fil^i(M) \otimes_R R_g$ is inside $M_g$ and $\Fil^i(\SM) \otimes_R R_g$ since $\Fil^i(M)$ is a subset of $\Fil^i(\SM)$ by definition, $\Fil^i(M) \subset M$ and $R_g$ is a flat $R$-module. Thus we have:  
\begin{align*}
\Fil^i(M) \otimes_R R_g \subset (\Fil^i(\SM) \otimes_R R_g) \cap M_g.
\end{align*}
But $\Fil^i(\SM) \otimes_R R_g$ as a subset of $\SM_g$ is in the kernel of $1 \otimes \varphi: \SM_g \to S_g/\Fil^iS_g \otimes_{\GS_g} \GM_g$, so we have $$\Fil^i(M) \otimes_R R_g \subset (\Fil^i(\SM_g)) \cap M_g = \Fil^i(M_g).$$

\end{proof}
\end{theorem}

Now that we have shown the compatibility of the filtrations, $M \otimes_R R_g$ has the full Fontaine-Laffaille module structure of $M_g$. We thus conclude:

\begin{theorem}
\label{genericFL}
$M \otimes_R R_g$ is a Fontaine-Laffaille module over $R_g$, and $M \otimes_R R_g \cong M_g$ as Fontaine-Laffaille modules. 
\begin{proof}
We established in Theorem \ref{Kisinbasechange} an isomorphism $M \otimes_R R_g \cong M_g$ compatible with Frobenius. We established in Theorem \ref{genericfil} that $\Fil^i(M) \otimes_R R_g \cong \Fil^i(M_g).$ Since $M_g$ carries the structure of a Fontaine-Laffaile module, we have completed the proof of this theorem.
\end{proof}
\end{theorem}
\end{subsection}
\begin{subsection}{$M \otimes_R W(k) \cong M_0$ as Fontaine-Laffaille Modules}

The strategy in this section is similar to the strategy in the above section where we analyze the Breuil module filtration to understand the Fontaine-Laffaille module filtration, but we have to deal with the additional concern that $W(k)$ is not a flat $R$-module, which means analyzing various terms of the form  $\mathrm{Tor}_1^R( -, W(k))$. These Tor terms aren't too concerning because the barrier to $W(k)$ being a flat $R$-module is, effectively, $t_i$-torsion, which does not show up in most of our modules.  \newline

\begin{lemma} 
\label{TorLemma} We have (for $i < p$):\newline
(a) $\Tor_1^{R}(S/R, W(k)) = 0$ \newline
(b) $\Tor_1^{R}(\SM/M, W(k)) = 0$ \newline
(c) $\Tor_1^{R}(S/\Fil^iS \otimes_{\GS} \GM, W(k)) = 0$\newline
(d) $\mathrm{Tor}_1^{R}(\SM/\Fil^i(\SM), W(k)) = 0$ \newline 
\begin{proof} 
For (a), we compute $\mathrm{Tor}^R_1(S/R, W(k)) = 0$ by recognizing that $$0 \rightarrow R \rightarrow S \rightarrow S/R \rightarrow 0$$ is a free resolution of $S/R$ since $S$ is a free $R$-module. When tensoring with $W(k)$, it is clear that the first map remains injective, so indeed $\mathrm{Tor}^R_1(S/R, W(k)) = 0$.  \newline 

For (b), by Lemma \ref{NakSec}, we have that $\SM/M$ is a direct sum of finitely many copies of $S/R$, so $\mathrm{Tor}^R_1(\SM/M, W(k)) = 0$ by part (a). \newline

For part (c), we observe that $S/\Fil^iS$ is (for $i < p$) a free $R$-module and hence a free $W(k)$-module. We also know $\GM$ is a free $\GS$-module, giving us our conclusion.\newline

For Part (d), we have an exact sequence $$0 \rightarrow \Fil^i(\SM) \hookrightarrow \SM \rightarrow S/\Fil^i(S) \otimes_{\GS} \GM.$$ Below in Lemma \ref{cokernelf}, we justify that the cokernel of this map 
 \newline$\SM \to S/\Fil^i(S) \otimes_{\GS}\GM$ is a finite free $R$-module, hence a free $W(k)$-module.  \newline

 Tensoring the induced sequence $$0 \to \SM/\Fil^i(\SM) \to S/\Fil^i(S) \otimes_{\GS} \GM$$ with $W(k)$ then applying Lemma \ref{cokernelf} and part (c) gives us the result. \newline
\end{proof}
\end{lemma}

The following lemma completes the proof of (c):

\begin{lemma}
\label{cokernelf}
Let $f: \SM \to S/\Fil^iS \otimes_{\GS} \GM$ be the map induced by $1 \otimes \varphi_{\GM}: \SM = S \otimes_{\varphi, \GS} \GM \to S \otimes_{\GS} \GM.$ Then $\mathrm{coker}(f)$ is a finite, free $R$-module. 
\begin{proof}
Write $\varphi^* \GM$ for $\GS \otimes_{\varphi, \GS} \GM$. We have an exact sequence $$0 \to \varphi^* \GM \to \GM \to \GM/\varphi^*\GM \to 0$$ Tensoring with $S$, we obtain $$S \otimes_{\varphi, \GS} \GM \to S \otimes_{\GS} \GM \to S \otimes_{\GS} \GM/\varphi^*\GM \to 0$$ where the first map is $1 \otimes \varphi_{\GM}.$ Then we obtain the following diagram projecting to the  $S/\Fil^i(S)$ level:

\adjustbox{scale=0.85,center}{%
  \begin{tikzcd}[scale=1]
    S \otimes_{\varphi, \GS} \GM \arrow{r}{1 \otimes \varphi_{\GM}} \arrow{dr}{f} & S \otimes_{\GS} \GM \arrow{r} \arrow{d} & S \otimes_{\GS} \GM/\varphi^*\GM \arrow{r} \arrow{d} & 0 \\
    & S/\Fil^iS \otimes_{\GS} \GM \arrow{r}{(*)} \arrow{d}& S/\Fil^iS \otimes_{\GS} \GM/\varphi^*\GM \arrow{d}& \\
    & 0 & 0 &
  \end{tikzcd}
} 

\noindent from which we conclude that $(*)$ is surjective and that $\mathrm{coker}(f) = S/\Fil^iS \otimes_{\GS} \GM/\varphi^*\GM.$ By Lemma \ref{quasiKisinclosed} and the proof of Lemma \ref{TorLemma}(c) which justifies that $\Tor_1^R(S/\Fil^iS, W(k)) = 0$ and allows us to commute the base change with the quotient, $\mathrm{coker}(f) \otimes_R W(k) \cong S_0/\Fil^iS_0 \otimes_{\GS_0} \GM_0/\varphi^*\GM_0.$ \newline

By tensoring $0 \to \varphi^*\GM_0 \to \GM_0 \to \GM_0/\varphi^*\GM \to 0$ with $S_0$, we find that $S_0/\Fil^iS_0 \otimes_{\GS_0} \GM_0/\varphi^*\GM_0$ is the cokernel of $f_0: \SM_0 \to S_0/\Fil^iS_0 \otimes_{\GS} \GM$ induced by $1 \otimes \varphi_{\GM_0}$. If $i < p$, by \cite[Theorem 4.20]{BuzzardDiamondJarvis} we know that $S_0/\Fil^iS_0 = W(k)[u]/E(u)^i$ and $\GM_0/\varphi^*\GM_0 = \displaystyle \oplus_i W(k)[u]/E(u)^{h_i}$ where $h_i$ are the Hodge-Tate weights from which it is clear that $S_0/\Fil^iS_0 \otimes_{\GS_0} \GM/\varphi^* \GM = \mathrm{coker}(f_0)$ is a finite free $W(k)$-module of rank, say, $m$. \newline 

Let $e_1, \cdots , e_m$ be a basis of $\mathrm{coker}(f_0)$ as a $W(k)$-module. By Nakayama's Lemma, we can lift it to $\widetilde{e_1}, \cdots , \widetilde{e_m}$ which generate $\mathrm{coker}(f)$ as an $R$-module. Analogously, we can see that $\mathrm{coker}(f) \otimes_R R_g \cong S_g/\Fil^i(S_g) \otimes_{\GS_g} \GM_g/\varphi^*\GM_g$ which, by the same reasoning replacing $W(k)$ with $W(k_g)$, is a free $W(k_g)$-module of rank $m$. The rank is the same since $\SM$ arises from a crystalline representation. A crystalline representation in the relative case is Hodge-Tate and will have the same Hodge-Tate weights on all of its fibers, so the decomposition of $\GM_0/\varphi^*\GM_0$ will use the same Hodge-Tate weights $h_i$ as the analogous decomposition of $\GM_g/\varphi^*\GM_g.$ Then the images of $\widetilde{e_1}, \cdots , \widetilde{e_m}$ in $\mathrm{coker}(f) \otimes_R R_g$ must generate $S_g/\Fil^i(S_g) \otimes_{\GS_g} \GM_g/\varphi^*\GM_g$. Therefore they are linearly independent over $W(k_g)$ and, hence, over $R$, as well. Thus $\mathrm{coker}(f)$ is a finite, free $R$-module. 
\end{proof}
\end{lemma}

On the Shilov point we had to be concerned that $S_g$ is very large: it is not true that $\SM \otimes_R R_g = \SM_g$, we would need an additional $p$-adic completion which does not behave nicely over non-Noetherian rings. On the closed fiber we have a little more freedom since $S_0$ is smaller than $S$ and this concern is not present on the closed fiber. We can thus obtain a stronger version of Lemma \ref{genericdecomp} more easily: 
\begin{lemma} 
We have $$\Fil^i(\SM_0) = \Fil^i\SM \otimes_R W(k).$$
\label{closeddecomp}
\begin{proof}
    We have the exact sequence $$0 \rightarrow \Fil^i\SM \rightarrow \SM \xrightarrow{f} (S/\Fil^iS) \otimes_{\GS} \GM \to \mathrm{coker}(f) \to 0$$ from the definition of $\Fil^i\SM$, and let $f$ be the labelled map. Factoring through the kernel, we obtain the short exact sequence $$0 \to \SM/\Fil^i\SM \to (S/\Fil^iS) \otimes_{\GS} \GM \to \mathrm{coker}(f) \to 0.$$ We now base change to $W(k)$, noting that $\mathrm{coker}(f)$ is a finite, free $R$-module by Lemma \ref{cokernelf} and the tensor produdct commutes with the quotients by Lemma \ref{TorLemma}. We thus obtain the exact sequence $$0 \to \SM_0/(\Fil^i\SM \otimes_{R} W(k))  \to (S_0/\Fil^iS_0) \otimes_{\GS_0} \GM_0 \to \mathrm{coker}(f) \otimes_R W(k) \to 0.$$ As $\Fil^i\SM_0$ is exactly the kernel of $\SM_0 \to (S_0/\Fil^iS_0) \otimes_{\GS_0} \GM_0$ inside of $\SM_0$ and $\Fil^i\SM \otimes_R W(k)$ is contained in this kernel, the proof of the lemma is complete. 
\end{proof}
\end{lemma}

Now that we can reduce the Breuil module filtration to the classical case, we will introduce a second filtration on $M$ which behaves a little more nicely under base change. In the classical case, the two filtrations agree, which is what we will show in our power series ring setting.

\begin{definition}
Let $$f_p: \SM \to M$$ be the map determined by $u \mapsto p$. Note this map sends $\frac{E(u)^i}{i!}$ to $0$. Define $\F^iM$ to be the image of $\Fil^i\SM$ under $f_p$.
\end{definition}

\begin{lemma}
$$\Fil^iM = \F^iM$$
\begin{proof}
We first show that $\Fil^iM \subset \F^iM$. Recall that we set $\mathscr{D} = \SM[p^{-1}]$ and $D = M[p^{-1}]$ upon which we can utilize the rational theory. Set $\Fil^iD \coloneqq \Fil^i\SD \cap D.$ We have the exact sequence $$0 \to \frac{\Fil^{i-1}\mathscr{D}}{(\Fil^pS)\SD} \xrightarrow{\cdot E} \frac{\Fil^i\mathscr{D}}{(\Fil^pS)\SD} \xrightarrow{f_p} \Fil^iD \to 0.$$ The kernel of $f_p: \Fil^i \SD \to \Fil^iD$ is $(\Fil^1S)\Fil^{i-1}\SD$ by \cite[Lemma 4.31]{Fcrystals}, and removing the $\Fil^pS$ part ensures that $\Fil^{i-1}\SD$ will give the entire kernel. Recalling from the rational theory that $$\Fil^i\SM = \Fil^i\mathscr{D} \cap \SM,$$ we obtain the exact sequence \begin{align}0 \to \frac{\Fil^{i-1}\SM}{(\Fil^pS)\SM} \to \frac{\Fil^i\SM}{(\Fil^pS)\SM} \xrightarrow{f_p} \F^iM \to 0.\end{align} Since $\Fil^iM \subset \Fil^i\SM \cap M \subset \Fil^i\SD \cap D$, we have an injective map $$\Fil^iM \hookrightarrow \Fil^iD \hookrightarrow \frac{\Fil^i\SD}{(\Fil^pS)\SD} \xrightarrow{f_p} \Fil^iD.$$ On the other hand, this agrees with the map $$\Fil^iM \hookrightarrow \frac{\Fil^i\SM}{(\Fil^pS)\SM} \xrightarrow{f_p} \F^iM \hookrightarrow \Fil^iD$$ giving us that $$\Fil^iM \subset \F^iM.$$

Tensoring $(1)$ with $W(k)$ and applying Lemma \ref{closeddecomp} we obtain the exact $$\frac{\Fil^{i-1} \SM_0}{(\Fil^pS_0)\SM_0} \to \frac{\Fil^i\SM_0}{(\Fil^pS_0)\SM_0} \xrightarrow{f_p} \F^iM \otimes_R W(k) \to 0.$$ But we also know from the classical theory that $\Fil^i(M_0)$ is the image of $\Fil^i\SM_0$ under $f_p$, so we obtain that $\Fil^iM_0 = \F^iM \otimes_R W(k)$ and hence $\F^iM \otimes_R W(k)$ is a finite, free $W(k)$-module. Let $\overline{e_1}, \cdots , \overline{e_a}$ be a basis and by Nakayama's Lemma we can lift it to a generating set $e_1, \cdots, e_a$ of $\F^iM$. Since $\F^iM$ is an image after $u \mapsto p$, it is clear that we can safely mod out $(1)$ by $\Fil^pS$ to obtain the exact $$0 \to \frac{\Fil^{i-1}\SM}{(\Fil^pS)\SM} \to \frac{\Fil^i\SM}{(\Fil^pS)\SM} \xrightarrow{f_p} \F^iM \to 0.$$ Now we can safely base change to the Shilov point to obtain $$0 \to \frac{\Fil^{i-1}\SM_g}{(\Fil^pS_g)\SM_g} \to \frac{\Fil^i\SM_g}{(\Fil^pS_g)\SM_g} \xrightarrow{f_p} \F^iM \otimes_R R_g \to 0.$$ the base change in the first two terms is as written due to Theorem \ref{genericFL} and the classical functor from Fontaine-Laffaille modules to Breuil modules described in Definition \ref{classicalfunctor}. Applying the classical theory to the Shilov point, we know $\F^iM \otimes_R R_g = \Fil^iM_g$ and hence is a finite, free $R_g$ module and is of the same rank as $\Fil^iM_0$. Note that the rank is the same again because of the Hodge-Tate property of crystalline representations. The places where the filtrations of $M_0$ and $M_g$ increase in rank and the amount they increase in rank by are determined completely by the Hodge-Tate weights of the crystalline representation associated to $M$. These are the same on each fiber and, consequently, the ranks of $\Fil^iM_0$ and $\Fil^iM_g$ are the same. \newline 

Thus with our embedding $\F^iM \hookrightarrow \Fil^iM \otimes_R R_g$, we see that $e_1, \cdots , e_a$ is a basis of $\F^iM$ and hence $\F^iM$ is a finite, free $R$-module.  \newline

Now we know that $\F^iM$ is a finite, free $R$-module contained in both $\Fil^iM_g$ and $\Fil^i\SD$ We conclude that \begin{align*} \F^iM &\subset (\F^iM \otimes_R R_g) \cap \Fil^i\SD \\
&= \Fil^iM_g \cap \Fil^i \SD\\
&= \Fil^iM
\end{align*}

\end{proof}
\end{lemma}

\begin{theorem} 
\label{closedFil}
$\Fil^i(M_0) \cong \Fil^i(M) \otimes_R W(k).$ 
\begin{proof}
In the proof of the previous lemma we established that $\F^iM \otimes_R W(k) = \Fil^iM_0$. The result of the previous lemma is that $\F^iM = \Fil^iM$, completing the proof. 
\end{proof}
\end{theorem}

This completes the proof of the following theorem:

\begin{theorem}
\label{closedFL}
$M \otimes_R W(k)$ is a Fontaine-Laffaille module over $W(k)$, and $M \otimes_R W(k) \cong M_0$ as Fontaine-Laffaille modules. 
\end{theorem}

\end{subsection}
\begin{subsection}{$M$ is an object of $\mathrm{MF}_{\nabla}^{\mathrm{ff}}$}
We now complete the proof of Theorem \ref{MisFL}:
\begin{proof}
Having established Theorem \ref{genericFL} and Theorem \ref{closedFL}, we have shown Condition (1) of Proposition \ref{genericclosed}. Since $M$ is a finite, free $R$-module, condition (2) holds, and thus $M$ is a Fontaine-Laffaille module over $R$ and belongs to the category $\mathrm{MF}^{\mathrm{ff}}(R)$.  
\end{proof}

\begin{theorem} $M$ is an object of the category $\mathrm{MF}^{\mathrm{ff}}_\nabla$. 
\begin{proof}
By Theorem \ref{closedFL} and Theorem \ref{quasiKisinTheorem}, we only need to show that $M$ is stable under $\nabla_{\GM}$ and that it satisfies integral Griffiths Transversality. \newline

Using Lemma \ref{section} and Lemma \ref{flatlemma} we see that $M = M[p^{-1}] \cap \SM$  as submodules of $\SM[p^{-1}].$ We know $M[p^{-1}]$ is stable under $\nabla_{\SM}$ (which restricts to $\nabla_{\GM}$ on $M[p^{-1}]$) and, by Proposition \ref{BreuilNablaStable}, so is $\SM$, showing that $M$ is stable under $\nabla_{\GM}.$ \newline

Because $\Fil^i(M) = \Fil^i(\SM) \cap M$, $M$ is stable under $\nabla_{\SM}$, and $\nabla_\SM$ satisfies integral $S$-Griffiths transversality by lemma \ref{BreuilNablaStable} we conclude $\nabla_{\GM}$ satisfies Griffiths transversality, as needed. 
\end{proof}
\end{theorem}
\end{subsection}
\begin{subsection}{Proof of Essential Surjectivity}
We now complete the proof that $\Tcris$ is essentially surjective. Let $\Lambda$ be an object of $\mathrm{Rep}_{\mathbb{Z}_p, [0,r]}^{\cris}(G_R)$. Let $(\GM, \varphi_{\GM}, f)$ be the Kisin module with descent data associated to $\Lambda$ guaranteed by Theorem \ref{quasiKisinTheorem} (note Remark \ref{dualremark}). Note also this theorem associates a Kisin module with descent data to an \'etale $(\varphi_{\etalering}, \etalering)$-module $\mathcal{M}$ with $T_{\etalering}(\mathcal{M})(r) \cong \Lambda$, and then $T_{\GS}^r(\GM) \cong \Lambda$ by Lemma \ref{etaleKisinGaloisEq}. \newline 

Let $\SM \cong \GM \otimes_{\varphi, \GS} S$ be the Breuil module associated with $\GM.$ We first prove the following stepping stone to Theorem \ref{TGTScompat}.
 \newline 

\begin{lemma}
\label{GaloisCongLemma} The natural injection $$T_{\GS}^r(\GM) \to T_S(\SM) \coloneqq \Fil^r(\AcrisR \otimes_S \SM)^{\varphi_r=1}$$ induced by the embedding $\AinfR \to \AcrisR$ is an isomorphism of $G_{\tilde{R}_\infty}$ representations.  
\begin{proof}
We prove the lemma by reducing to the Shilov point where we argue as in \cite[Lemma 4.7]{https://doi.org/10.48550/arxiv.2110.06001}. First note that using Lemma \ref{etaleKisinGaloisEq}, $$\Lambda \cong T_{\etalering}(\mathcal{M})(r) \cong T^r_{\GS}(\GM) \to (\Fil^r\varphi^*\GM \otimes_{\GS} \AinfR)^{\varphi_r = 1} \cong (\widehat{\mathcal{O}_{\mathcal{E}, g}^{\mathrm{ur}}} \otimes_{\mathcal{O}_{\mathcal{E}, g}} \mathcal{M}_g)^{\varphi = 1}(r) \cong \Lambda$$ is an isomorphism of $G_{\tilde{R}_{g,{\infty}}}$ representations. \newline

Then we obtain the diagram: 

 \adjustbox{scale=1,center}{%
  \begin{tikzcd}[scale=1]    
      T_{\GS}^r(\GM) \arrow{r} \arrow{d}{\cong} & T_S(\SM) \arrow{d}\\
      (\Fil^r\varphi^*\GM \otimes_{\GS} \Ainf)^{\varphi_r = 1} \arrow{r}{\cong} & \Fil^r(\Acris(R_g) \otimes_S \SM_g)^{\varphi_r=1}
  \end{tikzcd}
} 
where the bottom map is an isomorphism by \cite[Prop 6.12]{LiLiu} and $\Ainf$ represents the usual $\Ainf$ from the classical theory. We do not know whether the right vertical map is an injection since the natural map $\Acris(R) \to \Acris(R_g)$ requires a $p$-adic completion. We can, however, extend the diagram to the right because if $M$ is the Fontaine-Laffaille module associated to $\SM$, then the unique section of Lemma \ref{section} and the discussion of Section \ref{BreuilGaloisActionTwo} induces an isomorphism $T_S(\SM) \cong \Tcris(M)$, and we get a similar isomorphism on the Shilov point.  This gives us the diagram: 

 \adjustbox{scale=1,center}{%
  \begin{tikzcd}[scale=1]    
      T_{\GS}^r(\GM) \arrow{r} \arrow{d}{\cong} & T_S(\SM) \arrow{d} \arrow{r}{\cong} & \Tcris(M) \arrow{d}{\cong}\\
      (\Fil^r\varphi^*\GM \otimes_{\GS} \Ainf)^{\varphi_r = 1} \arrow{r}{\cong} & \Fil^r(\SM_g \otimes_S  \Acris(R_g) )^{\varphi_r=1} \arrow{r}{\cong} & \Tcris(M_g)
  \end{tikzcd}
} 
where the rightmost vertical arrow is an isomorphism by \cite[Cor 2.3.5]{LMP}. Since $T_{\GS}^r(\GM) \to T_S(\SM)$ is an injection, it must then be an isomorphism, completing the proof of the lemma.  
\end{proof}
\end{lemma}

The last diagram in the proof of this lemma then establishes that $\Tcris(M) \cong \Lambda$ as $G_{R_{\infty}}$ representations. We now are able to prove Theorem \ref{TGTScompat} by showing this is actually an isomorphism of $G_R$ representations which comes from the fact that the above maps are compatible with the relevant Galois actions. \newline 

The isomorphism $T_S(\SM) \to \Tcris(M)$ is determined by the unique $\nabla$-compatible section $M \to \SM$, and it is clear that the Galois action on $\Tcris(M)$ determined by $\nabla$ on $M$ is compatible with the Galois action on $T_S(\SM)$ determined by $\nabla$ on $\SM$, confirming that this map is actually an isomorphism of $G_R$ representations. \newline

To see that the map $T_{\GS}^r(\GM) \to T_S(\SM)$ is compatible with the $G_R$ action, \cite[Proposition 4.6]{Fcrystals} confirms that there is a unique descent datum $f$ on $\GM$ such that $f_S = \mathrm{id}_S \otimes_{\varphi, \GS} f$. As the descent data is used to construct the full $G_R$ action in the same way as described in Section \ref{BreuilGaloisActionTwo} and Section \ref{KisinGaloisAction} and this is the same as the Galois action defined via the $\nabla$ structure, the Galois action established on $T_{\GS}^r(\GM)$ is compatible with the Galois action on $T_S(\SM)$ and thus the map $T_{\GS}^r(\GM) \to T_S(\SM)$ which was previously shown to be an isomorphism of $G_{R_\infty}$ representations is actually an isomorphism of $G_R$ representations, completing the proof of Theorem \ref{TGTScompat}. \newline 

From the diagram in the above proof, we have also shown that $T_{\mathrm{cris}}(M) \cong T_{\GS}^r(\GM) \cong \Lambda$ as $G_R$ representations, which completes the proof of essential surjectivity over $R$. 
\end{subsection}
\end{section}

\begin{section}{Essential Surjectivity of $\Tcris$ over small base rings}
We will now prove theorems analaogous to those in the last section but in more generality over small base rings. While the results are analogous, the proof is easier because we are able to leverage the theory from the previous section to define a $\varphi_i$ structure via the rational theory, which would be difficult to do for the power series ring. Throughout this section, fix $T$ a small base ring. Write $\GS = \GS_T = T\llbracket u \rrbracket$ and $S = S_T$. Let $I_0 \subset S$ denote the kernel of the projection $S \to T$ given by $u \mapsto 0$.
\begin{subsection}{Statement of Main Theorems}
\begin{theorem}
\label{MisFLoverT}
Let $\SM$ be a Breuil module which arises from a Kisin module which arises from a $\Lambda,$ an object in $\mathrm{Rep}_{\mathbb{Z}_p}^{\cris [0,r]}(G_T).$  Set $$M \coloneqq \SM/I_0\SM \cong \GM/u\GM \otimes_{T, \varphi_{\GS}} T.$$ Then $M$ carries the data of a Fontaine-Laffaille module and satisfies $M \otimes_R S \cong \SM$ as in Definition \ref{classicalfunctor} We call it the ``Fontaine-Laffaille module associated to $\Lambda$."
\end{theorem}
\begin{theorem}
Let $M$ be the Fontaine-Laffaille module associated to $\Lambda.$ Then $$T_{\mathrm{cris}}(M) \cong \Lambda.$$ 
\end{theorem}

\end{subsection} 
\begin{subsection}{Preliminaries}
Similar to the previous section, our strategy will be to show the Fontaine-Laffaille data is compatible with the base change to $\Tatem \cong R$ discussed in Section \ref{SmallBaseChange} and then to apply Theorem \ref{TateLMPTheorem}. This time we will be able to rely on the rational theory combined with the theory we have established for a power series ring. \newline

Recall that associated to $V \coloneqq \Lambda \otimes_{\mathbb{Z}_p} \otimes \mathbb{Q}_p$ is $\Dcrisdual(V)$ (which we will denote with $D$ in this section), a finite, projective $T[p^{-1}]$ module equipped with a filtration structure $\Fil^iD$, a Frobenius structure $\varphi_D$, and a $\nabla_D$ structure. We will transfer the data on $D$ to $M$ in a similar way as how in the last section we transferred the data on $\SM$ to $M$. \newline

By Theorem \ref{Kisinbasechange}, we can see that $M \otimes_T \Tatem$ and $\Tcris(\Lambda|_{\Tatem})$, the Fontaine-Laffaille over $\Tatem$ associated to $\Lambda|_{\Tatem}$ are isomorphic as modules (note that $\Tatem$ is a local ring, and a projective module over a local ring is free). Write $\widehat{M}$ for $\Tcris(\Lambda|_{\Tatem})$ through the remainder of this section. 
\begin{lemma}
Consider the projection $q: \SM \to M$ induced by the $\varphi$-compatible projection $S \to T$, $u \mapsto 0$. Then $q$ admits a unique $\varphi$-compatible section $s: M \to \SM$. 

\begin{proof}
    The proof of Lemma \ref{sectionstable} works identically when the base ring is $T$. 
\end{proof}
\end{lemma}

This section gives us an injective map $M \hookrightarrow \SM \hookrightarrow \SD \coloneqq \SM[p^{-1}]$ and the projection $u \mapsto 0$ takes $\SD$ to $D$, and thus we have an injective map $M \hookrightarrow D.$ 
\end{subsection}
\begin{subsection}{The Filtration Structure}
We first define the filtration structure on $M$ via the filtration on $D$ from the rational theory.
\begin{definition}
Set $$\Fil^iM \coloneqq \Fil^iD \cap M.$$
\end{definition}
It remains to show this filtration behaves well under base change. 
\begin{lemma}
    $\Fil^iM \otimes_T \Tatem \cong \Fil^i\hatM$
\begin{proof}
    Note that \cite[Lemma 2.13]{Fcrystals} describes the compatibility of $D$ under base change. Since $\hatM$ is a Fontaine-Laffaille module we know  $\Dcrisdual(\Lambda|_{\Tatem}) \cong \hatM[p^{-1}]$ and $\Fil^i\hatM = \Fil^i\Dcrisdual(\Lambda|_{\Tatem}) \cap M.$ Using Lemma \ref{flatlemma} we have 
    \begin{align*}
        \Fil(M) \otimes_T \Tatem &= (\Fil^i D \cap M) \otimes_T \Tatem \\
        &= (\Fil^iD \otimes_T \Tatem) \cap (M \otimes_T \Tatem) \\
        &= \Fil^i(\Dcrisdual(\Lambda|_{\Tatem})) \cap \hatM \\
        &= \Fil^i \hatM.
    \end{align*}
\end{proof}
\end{lemma}
\end{subsection}
\begin{subsection}{The $\varphi_i$ structure}
We now define the $\varphi_i$ structure on $M$. Fix $\mathfrak{m}$ a maximal ideal of $T$ and again write $\hatM$ For $M \otimes \Tatem$ As $\hatM$ is a Fontaine-Laffaille module, there is a $\varphi_i$ structure on $\hatM$ which we will denote $\widehat{\varphi_i}$. Via our $\varphi$-compatible embedding $M \hookrightarrow \widehat{M}$, we can define a $\varphi_i$ structure on $M$ as $$\varphi_i \coloneqq \widehat{\varphi}_i\bigg|_{\Fil^iM}.$$ Let $\widehat{D} \coloneqq \hatM[p^{-1}] = \Dcrisdual(\Lambda|_{\Tatem})$ Note that $\widehat{\varphi_i} = \frac{\varphi_{\widehat{D}}}{p^i}$, and the compatibility of $\Dcrisdual$ with base change in \cite[Lemma 2.13]{Fcrystals} gives us that $\frac{\varphi_D}{p^i}$ acts the same on elements of $\Fil^iM \subset D$ and sends them into $D$. Thus $\varphi_i$ sends $\Fil^iM$ into $$D \cap \widehat{M} = M \otimes_T T[p^{-1}] \cap M \otimes_T \Tatem = M$$ by Lemma \ref{flatlemma}. Since the $\widehat{\varphi}_i$ satisfies $\widehat{\varphi_{i-1}} = p\widehat{\varphi}_i$, the same holds true for our new $\varphi_i$ structure. 
\end{subsection}
\begin{subsection}{The $\nabla$ Structure}
As $D$ comes equipped with a $\nabla_D$ structure, set $$\nabla_M \coloneqq\nabla_D\big|_M.$$ The proof that $M$ is stable under $\nabla_M$ is similar to the proof that $M$ is stable under the $\varphi_i$ structure. \newline

We can embed $M$ into $\widehat{M}$ which is stable under its $\nabla_{\hatM}$ structure, and we can also embed $M$ into $D$ which is stable under its $\nabla_D$ structure. As $D \cap \hatM = M$ as above, we see that $M$ is stable under $\nabla_M$.
\end{subsection}
\begin{subsection}{$\Tcris$ is essentially surjective over $T$.}
Combining the last few subsections, we have equipped $M$ with the data of an object in $\mathrm{MF}_{\mathrm{big}, \nabla}(T)$ which is compatible with each base change to $\Tatem$, completing the proof of Theorem \ref{MisFLoverT} by utilizing Theorem \ref{TateLMPTheorem}. It only remains to show that 
\begin{theorem}
    $\Tcris(M) \cong \Lambda.$
    \begin{proof}
We define a map $\Tcris(M) \to \Tcris(\hatM)$ via the functors to Galois representations from Kisin modules. As $\Tcris(M) = T_S(\SM)$, we obtain an injective map  $$T_{\GS}^r(\GM)  = (\Fil^r\GM \otimes \Ainf(T))^{\varphi_r=1} \hookrightarrow (\Fil^r \SM \otimes \Acris(T))^{\varphi_r = 1} = \Tcris(M)$$ and analogously for $\T_{\GS}^r(\GM \otimes_{\GS_T} \GS_R) \to \Tcris(\hatM)$. We obtain the diagram     

 \adjustbox{scale=1,center}{%
  \begin{tikzcd}[scale=1]    
      \Tcris(M) \arrow{r}  & \Tcris(\hatM) \\
      T_\GS^r(\GM) \arrow{r}{\cong} \arrow{u}& \T_{\GS}^r(\GM \otimes_{\GS_T} \GS_{\Tatem})\arrow{u}{\cong}
  \end{tikzcd}
} 
as we also have a map $\Tcris(M) \to \Tcris(\hatM)$ since $\Tcris(M) \subset \Tcris(D)$ and $\Tcris(\hatM) \subset \Tcris(\widehat{D})$ and $\Tcris(D) = \Tcris(\widehat{D}).$ \newline

The right vertical arrow in the diagram above is an isomorphism by the proof of Lemma \ref{GaloisCongLemma} and the bottom horizontal arrow is an isomorphism by Remark \ref{KisinIsomRem} The left vertical arrow is an injection. Thus the injective $\Tcris(M) \to \Tcris(\hatM)$ is forced to be an isomorphism, which quickly completes the proof.
    \end{proof}
\end{theorem}
\end{subsection}
\end{section}
\printbibliography
\end{document}